\documentclass{amsart}

\usepackage{amsfonts}
\usepackage{amssymb}
\usepackage{enumerate}
\usepackage{amsmath}
\usepackage{amsthm}
\usepackage{graphicx}
\usepackage[english]{babel}
\usepackage{hyperref}

\numberwithin{equation}{section}

\newtheorem{theorem}{Theorem}[section]
\newtheorem*{theorem*}{Theorem}
\newtheorem{lemma}[theorem]{Lemma}

\newtheorem{problem}[theorem]{Problem}

\newtheorem{corollary}[theorem]{Corollary}
\newtheorem{statement}[theorem]{Statement}
\newtheorem{fact}[theorem]{Fact}

\theoremstyle{definition}

\newtheorem{remark}[theorem]{Remark}
\newtheorem{definition}[theorem]{Definition}

\DeclareMathOperator{\inter}{int}

\DeclareMathOperator{\diam}{diam}
\DeclareMathOperator{\dist}{dist}

\DeclareMathOperator{\supp}{supp}
\DeclareMathOperator{\Var}{Var}

\DeclareMathOperator{\Sim}{sim}
\DeclareMathOperator{\Dim}{Dim}

\newcommand{\N}{\mathbb{N}}

\newcommand{\R}{\mathbb{R}}
\newcommand{\Z}{\mathbb{Z}}
\newcommand{\E}{\mathbb{E}}
\renewcommand{\P}{\mathbb{P}}

\newcommand{\iA}{\mathcal{A}}

\newcommand{\iH}{\mathcal{H}}
\newcommand{\iI}{\mathcal{I}}
\newcommand{\iJ}{\mathcal{J}}
\newcommand{\iK}{\mathcal{K}}

\newcommand{\iR}{\mathcal{R}}
\newcommand{\iF}{\mathcal{F}}

\newcommand{\iU}{\mathcal{U}}

\newcommand{\iV}{\mathcal{V}}
\newcommand{\iW}{\mathcal{W}}
\newcommand{\iX}{\mathcal{X}}
\newcommand{\iZ}{\mathcal{Z}}

\newcommand{\eps}{\varepsilon}
\begin{document}

\title{Uniform dimension results \\ for fractional Brownian motion}
\author{Rich\'ard Balka}
\address{Current affiliation: Department of Mathematics, University of British Columbia, and Pacific Institute for the Mathematical Sciences,
1984 Mathematics Road, Vancouver, BC V6T 1Z2, Canada}
\address{Former affiliations: Department of Mathematics, University of Washington, Box 354350, Seattle, WA 98195-4350, USA and Alfr\'ed R\'enyi Institute of Mathematics, Hungarian Academy of Sciences, PO Box 127, 1364 Budapest, Hungary}
\email{balka@math.ubc.ca}
\thanks{The first author was supported by the
Hungarian Scientific Research Fund grant no.~104178.}

\author{Yuval Peres}
\address{Microsoft Research, 1 Microsoft Way, Redmond, WA 98052, USA}
\email{peres@microsoft.com}



\begin{abstract}
Kaufman's dimension doubling theorem states that for a planar Brownian motion
$\{\mathbf{B}(t): t\in [0,1]\}$ we have
\[ \P(\dim \mathbf{B}(A)=2\dim A \textrm{ for all } A\subset [0,1])=1,\]
where $\dim$ may denote both Hausdorff dimension $\dim_H$ and packing dimension $\dim_P$.
The main goal of the paper is to prove similar uniform dimension results in the one-dimensional case.
Let $0<\alpha<1$ and let $\{B(t): t\in [0,1]\}$ be a fractional Brownian motion of Hurst index $\alpha$. For a
deterministic set $D\subset [0,1]$ consider the following statements:
\begin{enumerate}[(A)]
\item \label{H} $\P(\dim_H B(A)=(1/\alpha) \dim_H A \textrm{ for all } A\subset D)=1$,
\item \label{P} $\P(\dim_P B(A)=(1/\alpha) \dim_P A \textrm{ for all } A\subset D)=1$,
\item \label{HP} $\P(\dim_P B(A)\geq (1/\alpha) \dim_H A \textrm{ for all } A\subset D)=1$.
\end{enumerate}
We introduce a new concept of dimension, the modified Assouad dimension, denoted by $\dim_{MA}$.
We prove that $\dim_{MA} D\leq \alpha$ implies \eqref{H}, which enables us to reprove a restriction theorem of Angel, Balka, M\'ath\'e, and Peres.
We show that if $D$ is self-similar then \eqref{H} is equivalent to $\dim_{MA} D\leq \alpha$.
Furthermore, if $D$ is a set defined by digit restrictions then \eqref{H} holds iff
$\dim_{MA} D\leq \alpha$ or $\dim_H D=0$. The characterization of \eqref{H} remains open in general.
We prove that $\dim_{MA} D\leq \alpha$ implies \eqref{P} and they are equivalent provided that
$D$ is analytic. We show that \eqref{HP} is equivalent to $\dim_H D\leq \alpha$. This implies that if $\dim_H D\leq \alpha$ and
$\Gamma_D=\{E\subset B(D): \dim_H E=\dim_P E\}$, then
\[ \P(\dim_H (B^{-1}(E)\cap D)=\alpha \dim_H E \textrm{ for all } E\in \Gamma_D)=1.\]
In particular, all level sets of $B|_{D}$ have Hausdorff dimension zero almost surely.
\end{abstract}

\maketitle

\section{Introduction}

Let $\dim_H$, $\dim_P$, and $\dim_A$ respectively denote the Hausdorff, packing, and Assouad dimension,
see Section~\ref{s:prel} for the definitions. The first uniform dimension result was proved by Kaufman~\cite{Ka3}.
\begin{theorem*}[Kaufman's dimension doubling theorem]
Let $\{\mathbf{B}(t): t\in [0,1]\}$ be a planar Brownian motion. Then, almost surely, for all $A\subset [0,1]$ we have
\begin{equation*}  \dim_H \mathbf{B}(A)=2\dim_H A \quad \textrm{and} \quad \dim_P \mathbf{B}(A)=2\dim_P A.
\end{equation*}
\end{theorem*}
Note that the packing dimension result is not stated in \cite{Ka3}, but it follows easily from the proof, see also
\cite[Section~9.4]{MPe}. Here `uniform' means that the exceptional null probability event on which the theorem does not hold
is independent of $A$. Stronger uniform results for Haudorff and packing measures were obtained in \cite{PT}, for processes
with stationary, independent increments see \cite{HP}.
The theorem was generalized for higher dimensional fractional Brownian motion of appropriate parameters:
in case of Hausdorff dimension see Monrad and Pitt~\cite{MP}, while in case of packing dimension see Xiao~\cite{X2}.

\bigskip

Now we consider one-dimensional (fractional) Brownian motion.
McKean~\cite{Mc} proved the following theorem
for Brownian motion (case $\alpha=1/2$), and the general result for fractional Brownian motion was established by
Kahane~\cite[Chapter~18]{Kh}.

\begin{theorem*}[Kahane]
Let $0<\alpha<1$ and let $\{B(t): t>0\}$ be a fractional Brownian motion of Hurst index $\alpha$.
For a Borel set $A\subset [0,1]$, almost surely, we have
\[ \dim_H B(A)=\min\left\{1,(1/\alpha)\dim_H A \right\}.\]
\end{theorem*}

The zero set $\iZ$ of $B$ shows that this is not a uniform dimension result,
since $\dim_H \iZ=1-\alpha$ almost surely, see e.g.\ \cite[Chapter~18]{Kh}.

\bigskip

Kahane's theorem uniformly holds for almost all translates of all Borel sets.
Kaufman~\cite{Ka} proved the following for Brownian motion, and Wu and Xiao \cite{WX}
extended it to fractional Brownian motion.

 \begin{theorem*}[Wu and Xiao]
Let $0<\alpha<1$ and let $\{B(t): t>0\}$ be a fractional Brownian motion of Hurst index $\alpha$.
Then, almost surely, for all Borel sets $A\subset [0,1]$ we have
\[ \dim_H B(A+t)=\min\{1,(1/\alpha)\dim_H A\} \quad \textrm {for almost all } t>0.\]
\end{theorem*}

Now we turn to the case of packing dimension.
Somewhat surprisingly, the formula $\dim_P B(A)=\min\left\{1,(1/\alpha)\dim_P A \right\}$ does not hold.
In order to obtain a general formula for $\dim_P B(A)$ we need another notion of dimension.
Let $\Dim_{\alpha}$ denote the $\alpha$-dimensional packing dimension profile introduced by
Falconer and Howroyd in \cite{FaH0}.
The following theorem is due to Xiao~\cite{X1}.

\begin{theorem*}[Xiao]
Let $0<\alpha<1$ and let $\{B(t): t>0\}$ be a fractional Brownian motion of Hurst index $\alpha$.
For a Borel set $A\subset [0,1]$, almost surely, we have
\[\dim_P B(A)=(1/\alpha)\Dim_{\alpha} A.\]
\end{theorem*}

The zero set of a linear Brownian motion $B$ witnesses that
the analogue of Kaufman's dimension doubling theorem does not hold in dimension one. Let $\dim$ denote Hausdorff dimension or packing dimension.
If instead of $[0,1]$ we take a small enough set $D\subset [0,1]$, then $\dim B(A)=2\dim A$ may hold for all
$A\subset D$.
For example let $W$ be a linear Brownian which is independent of $B$ and let
$D=\iZ$ be the zero set of $W$. Then Kaufman's dimension doubling theorem implies that, almost surely,
for all $A\subset D$ we have
\[\dim B(A)=\dim (B,W)(A)=2\dim A.\]
Which property of $D$ ensures the above formula? The main goal of the paper is to fully
answer this question in case of packing dimension, and partially answer it in case of Hausdorff dimension. 

More generally, let $0<\alpha<1$ and let $\{B(t): t\in [0,1]\}$ be a fractional Brownian motion of Hurst index $\alpha$.
Let $D\subset [0,1]$ be a given deterministic set and consider the following statements:
\begin{enumerate}[(A)]
\item \label{iH} $\P(\dim_H B(A)=(1/\alpha) \dim_H A \textrm{ for all } A\subset D)=1$,
\item \label{iP} $\P(\dim_P B(A)=(1/\alpha) \dim_P A \textrm{ for all } A\subset D)=1$,
\item \label{iHP} $\P(\dim_P B(A)\geq (1/\alpha) \dim_H A \textrm{ for all } A\subset D)=1$.
\end{enumerate}

Our main problems can be stated as follows.

\begin{problem} \label{p:H} Characterize the sets $D$ for which \eqref{iH} holds.
\end{problem}

\begin{problem} \label{p:P} Characterize the sets $D$ for which \eqref{iP} holds.
\end{problem}

\begin{problem} \label{p:HP} Characterize the sets $D$ for which \eqref{iHP} holds.

\end{problem}
Note that $B$ is almost surely $\gamma$-H\"older continuous for all $\gamma<\alpha$, so almost surely,
for all $A\subset [0,1]$ we have
\[\dim_H B(A)\leq (1/\alpha) \dim_H A \quad \textrm{and} \quad  \dim_P B(A)\leq (1/\alpha) \dim_P A.\]
Therefore Problems~\ref{p:H} and \ref{p:P} are about to determine the sets $D$ for which the images $B(A)$ are of maximal dimension for
all $A\subset D$.

Clearly $\dim_H D\leq \alpha$ is necessary for \eqref{iH}, and $\dim_P D\leq \alpha$ is necessary for \eqref{iP}.
Somewhat surprisingly, the converse implications do not hold.
In order to fully solve Problem~\ref{p:P} for analytic sets $D$
and partially solve Problem~\ref{p:H} we introduce a new notion of dimension, the modified Assouad dimension.

\begin{definition} Let $X$ be a totally bounded metric space. For $x\in X$ and $r>0$ let
$B(x,r)$ denote the closed ball of radius $r$ around $x$. For $A\subset X$ let $N_r(A)$
denote the smallest number of closed balls of diameter $r$ required to cover $A$. For all $0<\eps<1$ define the \emph{$\eps$-Assouad dimension} as
\begin{align*} \dim^{\eps}_A X=\inf \{&\gamma: \exists \ C<\infty \textrm{ such that, for all } 0<r\leq r^{1-\eps}\leq R, \\
&\textrm{we have }  \sup \{N_r(B(x,R)): x\in X\}\leq C(R/r)^{\gamma}\}.
\end{align*}
Then the \emph{$\eps$-modified Assouad dimension} is defined by
\[\dim^{\eps}_{MA} X=\inf \left\{\sup_i \dim^{\eps}_{A} X_i: X=\bigcup_{i=1}^{\infty} X_i\right\}.\]
Finally, we define the \emph{modified Assouad dimension} as
\[\dim_{MA} X=\sup_{0<\eps<1} \dim^{\eps}_{MA} X=\lim_{\eps \to 0+} \dim^{\eps}_{MA} X.\]
\end{definition}

\begin{remark} Independently, F.~L\"u and L.~Xi~\cite{LX} introduced a similar concept,
the \emph{quasi-Assouad dimension} defined as
\[\dim_{qA} X=\lim_{\eps \to 0+} \dim^{\eps}_{A} X.\]
Their motivation came from studying quasi-Lipschitz mappings. Clearly we have $\dim_{MA} X\leq \dim_{qA} X$, and for sufficiently homogeneous spaces
(e.g.\ self-similar sets and sets defined by digit restrictions) the two notions coincide.
\end{remark}

\begin{fact} \label{f:3d} For every totally bounded metric space $X$ we have
\[\dim_H X\leq \dim_P X\leq \dim_{MA} X\leq \dim_A X.\]
\end{fact}

We prove a sufficient condition for \eqref{iH} and \eqref{iP}.

\begin{theorem} \label{t:main} Let $0<\alpha<1$ and let $\{B(t): t\in [0,1]\}$ be a fractional Brownian motion of Hurst index $\alpha$. Let
$D\subset [0,1]$ be such that $\dim_{MA} D\leq \alpha$. Then
\begin{align*} &\P(\dim_H B(A)=(1/\alpha)\dim_H A  \textrm{ for all } A\subset D)=1, \\
&\P(\dim_P B(A)=(1/\alpha)\dim_P A  \textrm{ for all } A\subset D)=1.
\end{align*}
\end{theorem}

Moreover, we answer Problem~\ref{p:P} for analytic sets $D$ as follows.

\begin{theorem} \label{t:p} Let $0<\alpha<1$ and let $\{B(t): t\in [0,1]\}$ be a fractional Brownian motion of Hurst index $\alpha$.
For an analytic set $D\subset [0,1]$ the following are equivalent:
\begin{enumerate}[(i)]
\item \label{equi1} $\P(\dim_P B(A)=(1/\alpha) \dim_P A \textrm{ for all } A\subset D)=1$,
\item \label{equi2} $\dim_{MA} D\leq \alpha$.
\end{enumerate}
\end{theorem}

The answer for Problem~\ref{p:HP} only uses the notion of Hausdorff dimension.

\begin{theorem} \label{t:mixed} Let $0<\alpha<1$ and let $\{B(t): t\in [0,1]\}$ be a fractional Brownian motion of Hurst index $\alpha$. For a set
$D\subset [0,1]$ the following are equivalent:
\begin{enumerate}[(i)]
\item \label{mix1} $\P(\dim_P B(A)\geq (1/\alpha) \dim_H A \textrm{ for all } A\subset D)=1$,
\item \label{mix2} $\dim_{H} D\leq \alpha$.
\end{enumerate}
\end{theorem}

\begin{corollary} Let $0<\alpha<1$ and let $\{B(t): t\in [0,1]\}$ be a fractional Brownian motion of Hurst index $\alpha$. Let $D\subset [0,1]$ with $\dim_H D\leq \alpha$ and define a random family of sets $\Gamma_D=\{E\subset B(D): \dim_H E=\dim_P E\}$. Then
\[\P(\dim_H (B^{-1}(E)\cap D)=\alpha \dim_H E \textrm{ for all } E\in \Gamma_D)=1.\]
In particular, all level sets of $B|_{D}$ have Hausdorff dimension zero almost surely.
\end{corollary}

\begin{remark} The condition $\dim_H D\leq \alpha$ is necessary even in the above corollary. Indeed,
if $D\subset [0,1]$ is compact with $\dim_H D>1/2$, then the zero set $\iZ$ of a linear Brownian motion satisfies
$\dim_H (D\cap \iZ)>0$ with positive probability, see \cite{Ka2}.
\end{remark}

We partially answer Problem~\ref{p:H} by considering sets defined by digit restrictions and
self-similar sets. The problem remains open in general.

\begin{definition} We define sets by restricting which digits can occur at a certain position of their dyadic expansion.
For $S\subset \N^{+}$ let $D_S\subset [0,1]$ be the compact set
\[D_S=\left\{\sum_{n=1}^{\infty} x_n 2^{-n}: x_n\in \{0,1\} \textrm{ if } n\in S \textrm{ and } x_n=0 \textrm{ if } n\notin S\right\}.\]
\end{definition}

\begin{theorem} \label{t:digit} Let $0<\alpha<1$ and let $\{B(t): t\in [0,1]\}$ be a fractional Brownian motion of Hurst index $\alpha$.
For every $S\subset \N^{+}$ the following are equivalent:
\begin{enumerate}[(i)]
\item \label{digi1} $\P(\dim_H B(A)=(1/\alpha) \dim_H A \textrm{ for all } A\subset D_S)=1$,
\item \label{digi2} $\dim_{MA} D_S\leq \alpha$ or $\dim_H D_S=0$.
\end{enumerate}
\end{theorem}

\begin{definition} A compact set $D\subset \R^d$ is called \emph{self-similar}
if there is a finite set $\{F_i\}_{i\leq k}$ of contracting similarities of $\R^d$
such that $D=\bigcup_{i=1}^{k} F_i(D)$.
\end{definition}

Let $D\subset [0,1]$ be a self-similar set. Recently Fraser et al.\ \cite[Theorem~1.3]{FHOR} proved that if $D$ satisfies the so-called weak separation property (a weakening of the open set condition), then $\dim_H D=\dim_A D$; otherwise $\dim_A D=1$.
Hence Theorem~\ref{t:main} and Fact~\ref{f:3d} yield that if $D$ is a self-similar set
with the weak separation property then \eqref{iH} holds iff $\dim_{MA} D\leq \alpha$.
We prove that this remains true regardless of separation conditions.

\begin{theorem} \label{t:ss} Let $0<\alpha<1$ and let $\{B(t): t\in [0,1]\}$ be a fractional Brownian motion of Hurst index $\alpha$.
For a self-similar set $D\subset [0,1]$ the following are equivalent:
\begin{enumerate}[(i)]
\item \label{i:ss1} $\P(\dim_H B(A)=(1/\alpha) \dim_H A \textrm{ for all } A\subset D)=1$,
\item \label{i:ss2}  $\dim_{MA} D\leq \alpha$.
\end{enumerate}
\end{theorem}

Falconer~\cite[Theorem~4]{F} proved that $\dim_H D=\dim_P D$
for every self-similar set $D$. First Fraser~\cite[Section~3.1]{Fr} constructed a self-similar set $D\subset [0,1]$
for which $\dim_H D<\dim_A D$, answering a question of Olsen~\cite[Question~1.3]{O}. A positive answer to the following problem
together with Theorem~\ref{t:main} would immediately imply Theorem~\ref{t:ss}.

\begin{problem} Does $\dim_H D=\dim_{MA} D$ hold for all self-similar sets $D\subset [0,1]$?
\end{problem}

The following restriction theorem for fractional Brownian motion is due to Angel, Balka, M\'ath\'e, and Peres~\cite{ABMP}. As an application of our theory, we give a new proof for this result based on Theorem~\ref{t:main}.

\begin{theorem}[Angel et al.] \label{t:ABMP}
Let $0<\alpha<1$ and let $\{B(t): t\in [0,1]\}$ be a fractional Brownian motion of Hurst index $\alpha$.
Then, almost surely, for all $A\subset [0,1]$ if $B|_{A}$ is
$\beta$-H\"older continuous for some $\beta>\alpha$ then $\dim_H A\leq 1-\alpha$.
\end{theorem}

In fact, in \cite{ABMP} a stronger form of the above theorem was proved, where Hausdorff dimension was replaced by
upper Minkowski dimension. Theorem~\ref{t:ABMP} also implies the following, see \cite[Section~8]{ABMP} for the deduction.
%%%
\begin{theorem*}[Angel et al.] Let $0<\alpha<1$ and let $\{B(t): t\in [0,1]\}$ be a fractional Brownian motion of Hurst index $\alpha$.
Then, almost surely, for all $A\subset [0,1]$ if $B|_{A}$ is of bounded variation then $\dim_H A\leq \max\{1-\alpha,\alpha\}$.
In particular,
\[\mathbb{P}(\exists A: \dim_H A>\max\{1-\alpha,\alpha\} \textrm{ and } B|_{A} \textrm{ is non-decreasing})=0.\]
\end{theorem*}

Let $\iZ$ be the zero set of $B$ and let $\iR=\{t\in [0,1]: B(t)=\max_{s\in [0,t]} B(s)\}$ denote the set of
record times of $B$. Then, almost surely, $\dim_H \iZ=1-\alpha$ and $\dim_H \iR=\alpha$,
see e.g.\ \cite[Chapter~18]{Kh} and \cite{ABMP}, respectively. Therefore $\iZ$ and $\iR$ witness that the above theorems are best possible.

\begin{remark} Theorem~\ref{t:ABMP} can be generalized by replacing Hausdorff dimension by quasi-Assouad dimension,
 the proof in \cite{ABMP} works verbatim. This yields that $\dim_H \iZ=\dim_{MA} \iZ=\dim_{qA} \iZ=1-\alpha$ almost surely.
 The proof of \cite[Proposition~1.5]{ABMP} readily implies that $\dim_H \iR=\dim_{MA} \iR=\dim_{qA} \iR=\alpha$ with probability one.
\end{remark}

\bigskip

In Section~\ref{s:prel} we outline the definitions of our main notions and some of
their basic properties. Section~\ref{s:heart} is devoted to the proof of the technical Theorem~\ref{t:heart},
which will be the basis of the proofs of Theorems~\ref{t:main} and \ref{t:mixed}. In Section~\ref{s:main} we prove our main results,
Theorems~\ref{t:main}, \ref{t:p}, and \ref{t:mixed}. In Section~\ref{s:examples} we prove Theorems~\ref{t:digit} and ~\ref{t:ss}, which answer Problem~\ref{p:H} in case of sets defined by digit restrictions and self-similar sets, respectively.
In Section~\ref{s:restriction} we reprove Theorem~\ref{t:ABMP} by applying Theorem~\ref{t:main}. In order to do so, we show that a
percolation limit set has equal Hausdorff and modified Assouad dimension almost surely.

\section{Preliminaries} \label{s:prel}

Let $X$ be a totally bounded metric space. Assume that $x\in X$, $r>0$, and $A\subset X$.
Recall that $B(x,r)$ denotes the closed ball of radius $r$ around $x$, and $N_r(A)$ is
the smallest number of closed balls of diameter $r$ required to cover $A$.
The diameter and interior of $A$ is denoted by $\diam A$ and $\inter A$, respectively.
For all $s\geq 0$ the \emph{$s$-Hausdorff content} of $X$ is defined as
\begin{equation*}
\mathcal{H}^{s}_{\infty}(X)=\inf \left\{\sum_{i=1}^\infty (\diam X_{i})^{s}: X=\bigcup_{i=1}^{\infty} X_{i}\right\}.
\end{equation*}
We define the \emph{Hausdorff dimension} of $X$ by
\begin{equation*}
\dim_H X = \inf\{s \ge 0: \mathcal{H}^{s}_{\infty}(X)=0\}.
\end{equation*}
The \emph{lower and upper Minkowski dimensions} of $X$ are respectively defined as
\begin{align*} \underline{\dim}_{M} X&=\liminf_{r\to 0+} \frac{\log N_r(X)}{-\log r}, \\
\overline{\dim}_M  X&=\limsup_{r\to 0+} \frac{\log N_r(X)}{-\log r}.
\end{align*}
Equivalently, the upper Minkowski dimension of $X$ can be written as
\[\overline{\dim}_M  X=\inf \{\gamma: \exists C<\infty \textrm{ such that }
N_r(X)\leq Cr^{-\gamma} \textrm{ for all } r>0\}.\]
We define the \emph{packing dimension} of $X$ as the modified upper Minkowski dimension:
\[\dim_{P} X=\inf \left\{\sup_i \overline{\dim}_{M} X_i: X=\bigcup_{i=1}^{\infty} X_i\right\}.\]
The \emph{Assouad dimension} of $X$ is given by
\begin{align*} \dim_{A} X=\inf \{&\gamma: \exists C<\infty \textrm{ such that, for all } 0<r\leq R, \\
&\textrm{we have }  \sup\{N_r(B(x,R)): x\in X\}\leq C(R/r)^{\gamma}\}.
\end{align*}

For more on these concepts see \cite{F2} or \cite{Ma}.

\begin{fact} For every totally bounded metric space $X$ we have
\begin{align*} \dim_H X&\leq \underline{\dim}_M X\leq \overline{\dim}_M X \leq \dim_A X, \\
\dim_H X&\leq \dim_P X\leq \dim_{MA} X\leq \dim_A X.
\end{align*}
\end{fact}

\begin{proof} The inequalities in the first row and $\dim_H X\leq \dim_P X$ are well known, see e.g.\ \cite{F2}.
Clearly for all $\eps\in (0,1)$ we have $\overline{\dim}_M X\leq \dim^{\eps}_{A} X\leq \dim_A X$,
thus $\dim_P X\leq \dim_{MA} X\leq \dim_A X$.
\end{proof}

A separable, complete metric space is called a \emph{Polish space}.
A separable metric space $X$ is \emph{analytic}
if there exists a Polish space $Y$ and a continuous
onto map $f\colon Y\to X$. For more on this concept see \cite{Ke}.
The proof of the following theorem is a modification of the proof of \cite[Theorem~2]{Ha}.

\begin{theorem} \label{t:reg} Let $X$ be a totally bounded analytic metric space. Then
\[\dim_{MA} X=\sup\{\dim_{MA}K: K\subset X \textrm{ is compact}\}.\]
\end{theorem}

The following lemma is classical. For part \eqref{eq:pack1} see the proof of \cite[Proposition~3]{T} or \cite[Corollary~3.9]{F2},
for part \eqref{eq:pack2} see \cite[Lemma~3.2]{MM} or \cite[Lemma~4]{FaH}.

\begin{lemma}  \label{l:pack} Let $X$ be a totally bounded metric space.
\begin{enumerate}[(i)]
\item \label{eq:pack1} If $X$ is compact and if $\overline{\dim}_{M} U\geq s$
for every non-empty open set $U\subset X$, then $\dim_P X\geq s$.
\item \label{eq:pack2}
If $\dim_P X>s$, then there is a closed set $C\subset X$ such that $\dim_P (C\cap U)>s$ for every
open set $U$ which intersects $C$.
\end{enumerate}
\end{lemma}

The proof of the following lemma is similar to that of Lemma~\ref{l:pack}.
For the sake of completeness we outline the proof.

\begin{lemma}  \label{l:homog} Let $X$ be a totally bounded metric space and let $0<\eps<1$.
\begin{enumerate}[(i)]
\item \label{eq:homog1} If $X$ is compact and if $\dim^{\eps}_{A} U\geq s$ for every non-empty open set $U\subset X$, then $\dim^{\eps}_{MA} X\geq s$.
\item \label{eq:homog2}
If $\dim^{\eps}_{MA} X>s$, then there is a closed set $C\subset X$ such that $\dim^{\eps}_{MA} (C\cap U)>s$ for every
open set $U$ which intersects $C$.
\end{enumerate}
\end{lemma}

\begin{proof} \eqref{eq:homog1} Assume that $X=\bigcup_{i=1}^{\infty} X_i$, where $X_i$ are closed subsets of $X$. Clearly a set and its closure
have the same $\eps$-Assouad dimension, so it is enough to prove that $\dim^{\eps}_{A} X_i\geq s$ for some $i$.
 By Baire's category theorem there is a non-empty open set $U$ in $X$ such that $U\subset X_i$ for some index $i$. Therefore $\dim^{\eps}_A X_i\geq \dim^{\eps}_A U\geq s$, which completes the proof of \eqref{eq:homog1}.

\eqref{eq:homog2} Let $\iV$ be a countable open basis for $X$. Define
\[C=X\setminus \bigcup \{V\in \iV: \dim^{\eps}_{MA} V\leq s\}.\]
Clearly $C$ is closed in $X$ and the $\eps$-modified Assouad dimension is countably stable.
Therefore $\dim^{\eps}_{MA} (X\setminus C)\leq s$.
Let $U\subset X$ be an open set intersecting $C$ and assume to the contrary that $\dim^{\eps}_{MA} (C\cap U)\leq s$.
Then there exists a $V\in \iV$ such that $V\subset U$ and $V\cap C\neq \emptyset$. Using the stability again,
we obtain that
\begin{align*}
\dim^{\eps}_{MA} V&\leq \max\{\dim^{\eps}_{MA} (X\setminus C), \dim^{\eps}_{MA} (C\cap V)\} \\
&\leq \max\{s,\dim^{\eps}_{MA} (C\cap U)\}=s.
\end{align*}
This contradicts the construction of $C$, so the proof of \eqref{eq:homog2} is complete.
\end{proof}

For $D\subset [0,1]$ and $\gamma\in (0,1]$ a function $f\colon D\to \R$ is
called \emph{$\gamma$-H\"older continuous} if there is a finite constant $C$ such that
$|f(x)-f(y)|\leq C|x-y|^{\gamma}$ for all $x,y\in D$. The minimum of such numbers
$C$ is the \emph{H\"older constant} of $f$.

\begin{fact} \label{f:Holder} Let $D\subset [0,1]$ and let
$f\colon D \to \R$ be a $\gamma$-H\"older continuous function for some $\gamma>0$.
Then for all $A\subset D$ we have
\[\dim_H f(A)\leq (1/\gamma) \dim_H A \quad \textrm{and} \quad \dim_P f(A)\leq (1/\gamma) \dim_H A.\]
\end{fact}

For $0<\alpha<1$ the process $\{B(t):t\geq 0\}$ is called a
\emph{fractional Brownian motion of Hurst index $\alpha$} if
\begin{itemize}
\item $B$ is a Gaussian process with stationary increments;
\item $B(0)=0$ and $t^{-\alpha}B(t)$ has standard normal distribution for every $t>0$;
\item almost surely, the function $t\mapsto B(t)$ is continuous.
\end{itemize}
The covariance function of $B$ is $\E(B(t)B(s))=(1/2)(|t|^{2\alpha}+|s|^{2\alpha}-|t-s|^{2\alpha})$.
It is well known that almost surely $B$ is $\gamma$-H\"older continuous for
all $\gamma<\alpha$, see Lemma~\ref{l:ub} below. For more information on
fractional Brownian motion see \cite{Ad}.

Let $|A|$ denote the cardinality of a set $A$.

\section{A key theorem} \label{s:heart}

The main goal of this section is to prove Theorem~\ref{t:heart}.
First we need some definitions. Assume that $0<\alpha<1$ is fixed and $\{B(t): t\in [0,1]\}$
is a fractional Brownian motion of Hurst index $\alpha$.

\begin{definition}  For $n\in \N$ and $p\in \{0,\dots, 2^n-1\}$ a \emph{dyadic time interval of order $n$} is of the form
\[ I_{n,p} = [p 2^{-n}, (p+1)2^{-n}].\]
For all $n\in \N$ let
\[\iI_n=\{I_{n,p}: 0\leq p<2^n\} \quad \textrm{and} \quad \iI^{n}=\bigcup_{k=n}^{\infty} \iI_k\]
be the set of dyadic time intervals of order exactly $n$ and at least $n$, respectively. Let
\[\iI=\bigcup_{n=0}^{\infty} \iI_n\]
denote the set of all dyadic time intervals in $[0,1]$. For $q\in \Z$ a \emph{value interval of order $n$} is of the form
\[
J_{n,q} = [q 2^{-\alpha n}, (q+1) 2^{-\alpha n}].
\]
For all $n\in \N$ let
\[\iJ_n=\{J_{n,q}: q\in \Z\}\]
be the set of value intervals of order $n$.
\end{definition}

\begin{definition} Let $\iU\subset \iI$ be a set of dyadic time intervals.
For all $I\in \iI$ and $n\in \N$ define \[N_n(\iU,I)=|\{U\in \iU\cap \iI_n: U\subset I\}|.\]
For $m,n\in \N$ with $m<n$ let
\[N_{m,n}(\iU)=\max\{N_n(\iU,I): I\in \iI_m\}.\]
Let $\beta>0$ and $\eps \geq 0$, then $\iU$ is called \emph{$(\beta,\eps)$-balanced} if for all $m\leq (1-\eps)n$ we have
\[N_{m,n}(\iU)\leq 2^{\beta(n-m)}.\]
We say that $\iU$ is \emph{$\beta$-balanced} if it is $(\beta,0)$-balanced.
For all $n\in \N$ and $q\in \Z$ define
\[G_{n,q}(\iU)=|\{U\in \iU\cap \iI_n: B(U)\cap J_{n,q}\neq \emptyset\}|.\]
For $\eps>0$ define the event
\[\Gamma(\iU,\eps)=\{G_{n,q}(\iU)\leq 2^{\eps n} \textrm{ for all } n\in \N \textrm{ and } q\in \Z\}.\]
\end{definition}

\begin{theorem} \label{t:heart} Let $0<\alpha<1$ and let $\{B(t): t\in [0,1]\}$ be a fractional Brownian motion of Hurst index $\alpha$.
Let $\eps>0$ and assume that $\iU_k\subset \iI^k$ are $(\alpha+\eps,\eps)$-balanced for all large enough $k$. Then, almost surely,
$\Gamma(\iU_k,3\eps)$ holds for all $k$ large enough.
\end{theorem}

Before proving the theorem we need some preparation.

\begin{definition} Let $\iU\subset \iI$.
For all $n\in \N$ and $q\in \Z$ define \[P_{n,q}(\iU)=|\{U\in \iU\cap \iI_n: B(\min U)\in J_{n,q}\}|.\]
For $\eps>0$ define the event
\[\Pi(\iU,\eps)=\{P_{n,q}(\iU)\leq 2^{\eps n} \textrm{ for all } n\in \N \textrm{ and } |q|\leq n2^{\alpha n}\}.\]
%%%%
Assume that our fractional Brownian motion $B$ is defined on the probability space $(\Omega, \iF, \P)$,
and let $\iF_t=\sigma(B(s): 0\leq s\leq t)$ be the natural filtration. For a stopping time $\tau\colon \Omega \to [0,\infty]$
define the $\sigma$-algebra
\[\iF_{\tau}=\{A\in \iF: A\cap \{\tau\leq t\}\in \iF_t \textrm{ for all } t\geq 0\}.\]
For all $n\in \N$, $q\in \Z$, and for all stopping times $\tau$ let
\[P_{n,q}^{\tau}(\iU)=|\{U\in \iU\cap \iI_n: \min U>\tau  \textrm{ and } B(\min U)\in J_{n,q}\}|.\]
\end{definition}

\begin{lemma} \label{l:sln} Let $\iU\subset \iI$ be $(\alpha+\eps,\eps)$-balanced for some $\eps>0$.
Then there is a finite constant $c=c(\alpha,\eps)$ such that for all bounded
stopping times $\tau$ and integers $n\in \N$ and $q\in \Z$, almost surely, we have
\[\E(P_{n,q}^{\tau}(\iU) \, | \, \iF_{\tau})\leq c 2^{\eps n}.\]
\end{lemma}

\begin{proof} We may assume that $\tau$ takes values from $2^{-n} \mathbb{N}$. Pitt \cite[Lemma~7.1]{P} proved that the property of
\emph{strong local nondeterminism} holds for fractional Brownian motion, that is, there is a constant
$c_1=c_1(\alpha)>0$ such that for all $t\geq 0$, almost surely, we have
\begin{equation} \label{eq:tau} \Var(B(\tau+t) \, | \, \iF_{\tau})\geq c_1 t^{2\alpha}.\end{equation}
Let us fix $t\in [2^{-m},2^{-m+1}]$ for some $m\in \N^+$. As $B$ is Gaussian, almost surely the conditional distribution $B(\tau+t)\,|\,\iF_{\tau}$ is normal,
and \eqref{eq:tau} implies that its density function is bounded by $1/(\sqrt{c_1}t^{\alpha})$. Therefore, almost surely, we have
\begin{equation} \label{eq:tau2}
\P(B(\tau+t)\in J_{n,q}  \, | \, \iF_{\tau})\leq \int_{q2^{-\alpha n}}^{(q+1)2^{-\alpha n}}
\frac{1}{\sqrt{c_1}t^{\alpha}} \,\mathrm{d} x=c_2(t2^{n})^{-\alpha}\leq c_2 2^{\alpha(m-n)},
\end{equation}
where $c_2=1/\sqrt{c_1}$.
Fix $n\in \N^+$ and for all $1\leq m\leq n$ define
\[\iX_m=\{\min U: U\in \iU\cap \iI_n \textrm{ and } U\subset [\tau+2^{-m},\tau+2^{-m+1}]\},\]
and let $X_m$ be the contribution of $B|_{\iX_m}$ to
$P_{n,q}^{\tau}(\iU)$. Clearly for all $m$ we have
\begin{equation} \label{eq:2nm}
\E(X_m \, | \, \iF_{\tau})\leq |\iX_m|\leq 2^{n-m}.
 \end{equation}
Since $[\tau+2^{-m},\tau+2^{-m+1}]\cap [0,1]$ can be always covered by two intervals of $\iI_m$ and $\iU$ is $(\alpha+\eps,\eps)$-balanced,
for all $m\leq (1-\eps)n$ we obtain that
\begin{equation} \label{eq:iY} |\iX_m|\leq 2^{(\alpha+\eps)(n-m)+1}.
\end{equation}
Applying \eqref{eq:tau2} and \eqref{eq:iY} yields that for all $m\leq (1-\eps)n$ we have
\begin{equation*} \E(X_m \, | \, \iF_{\tau})=\sum_{s\in \iX_m} \P(B(s)\in J_{n,q}  \, | \, \iF_{\tau})\leq 2^{(\alpha+\eps)(n-m)+1} c_2 2^{\alpha(m-n)}=c_3 2^{\eps (n-m)},
\end{equation*}
where $c_3=2c_2$. Thus the above inequality and \eqref{eq:2nm} imply that
\begin{align*} \E(P_{n,q}^{\tau}(\iU) \, | \, \iF_{\tau})&=\sum_{m=1}^n \E(X_m \, | \, \iF_{\tau}) \\
&\leq \sum_{1\leq m\leq (1-\eps)n} c_3 2^{\eps (n-m)}+\sum_{(1-\eps)n< m\leq n} 2^{n-m} \\
&\leq \frac{c_{3}}{2^{\eps}-1} 2^{\eps n}+2^{\eps n+1}\leq c 2^{\eps n}
\end{align*}
for some finite constant $c=c(\alpha,\eps)$. The proof is complete.
\end{proof}

\begin{lemma} \label{l:Snmpq} Let $\iU\subset \iI$ be $(\alpha+\eps,\eps)$-balanced for some $\eps>0$.
Then there is a finite constant $C=C(\alpha,\eps)$ such
that for all $n,\ell \in \N$ and $q\in \Z$ we have
\[\P(P_{n,q}(\iU)\geq  C \ell 2^{\eps n})\leq 2^{-\ell}.\]
\end{lemma}

\begin{proof}  Let $c=c(\alpha, \eps)$ be the finite constant in Lemma~\ref{l:sln}, clearly we may assume that $c\geq 1$.
We will show that $C=3c$ satisfies the lemma. We define stopping times $\tau_0,\dots, \tau_\ell$. Let $\tau_0=0$.
If $\tau_{k}$ is defined for some $0\leq k<\ell$ then
let $\tau_{k+1}$ be the first time such that $P_{n,q}^{\tau_k}(\mathcal{U})-P_{n,q}^{\tau_{k+1}}(\mathcal{U})\geq 2c 2^{\varepsilon n}$ if such a time exists, otherwise let $\tau_{k+1}=1$.
Then $c\geq 1$ and the definition of stopping times yield that
%%%
\begin{align*} \P(P_{n,q}(\iU)\geq  3\ell c 2^{\eps n})&\leq \P(P_{n,q}(\iU)\geq  \ell(2c 2^{\eps n}+1)) \\
&\leq \P(\tau_\ell<1)=\prod_{k=1}^{\ell} \P(\tau_{k}<1 \,|\, \tau_{k-1}<1).
\end{align*}
%%%
We may suppose that $\P(\tau_\ell<1)>0$ and thus the above conditional probabilities are defined,
otherwise we are done. Hence it is enough to show that for all $1\leq k\leq \ell$ we have
\begin{equation} \label{eq:1/2} \P(\tau_{k}<1 \,|\, \tau_{k-1}<1)\leq 1/2.
\end{equation}
Using Lemma~\ref{l:sln} and the conditional Markov's inequality, we obtain that, almost surely,
\[\P(\tau_k<1 \,|\, \iF_{\tau_{k-1}})= \P(P_{n,q}^{\tau_{k-1}}(\iU)\geq 2c 2^{\eps n}\, | \, \iF_{\tau_{k-1}})\leq 1/2.\]
As $\{\tau_{k-1}<1\}\in \iF_{\tau_{k-1}}$, the tower property of conditional expectation yields \eqref{eq:1/2}. This completes the proof.
\end{proof}

\begin{lemma}\label{l:dkey} Let $\eps>0$ and assume that $\iU_k\subset \iI^k$ are $(\alpha+\eps,\eps)$-balanced for all large enough $k$.
Then, almost surely, $\Pi(\iU_k,2\eps)$ holds for all $k$ large enough.
\end{lemma}

\begin{proof} We may assume that $\iU_k$ are $(\alpha+\eps,\eps)$-balanced for all $k$.
Let $C=C(\alpha,\eps)$ be the constant in Lemma~\ref{l:Snmpq}. We give an upper bound for the probability of the complement of $\Pi(\iU_k,2\eps)$. We apply Lemma~\ref{l:Snmpq} for all $n\geq k$ and $|q|\leq n2^{\alpha n}$ with $\ell=n^2$. Clearly $\iU_k\subset \iI^k$ implies
that $P_{n,q}(\iU_k)=0$ for all $n<k$ and $q\in \Z$.
As the number of integers $q$ with  $|q|\leq n2^{\alpha n}$ is at most $2n2^{\alpha n}+1<3n2^{\alpha n}$, for all large enough $k$ we obtain
\begin{align*} \P(\Pi^c(\iU_k,2\eps))&=\P(P_{n,q}(\iU_k)>2^{2\eps n} \textrm{ for some } n\geq k \textrm{ and } |q|\leq n2^{\alpha n}) \\
&\leq \sum_{n=k}^{\infty} \sum_{|q|\leq n2^{\alpha n}}  \P(P_{n,q}(\iU_k)>2^{2\eps n})\\
&\leq \sum_{n=k}^{\infty} \sum_{|q|\leq n2^{\alpha n}} \P(P_{n,q}(\iU_k)>Cn^22^{\eps n}) \\
&  \leq \sum_{n=k}^{\infty} (3n2^{\alpha n}) 2^{-n^2}\leq 2^{-k}.
\end{align*}
Thus $\sum_{k=1}^{\infty} \P(\Pi^c(\iU_k,2\eps))<\infty$, so the Borel-Cantelli lemma implies that, almost surely,
$\Pi^c(\iU_k,2\eps)$ holds only for finitely many $k$. This completes the proof.
\end{proof}

For the following well-known lemma see the more general \cite[Corollary~7.2.3]{MR}.

\begin{lemma} \label{l:ub} Almost surely, we have
\[\limsup_{h\to 0+} \sup_{0\leq t\leq 1-h} \frac{|B(t+h)-B(t)|}{\sqrt{2h^{2\alpha} \log(1/h)}}\leq 1.\]
\end{lemma}

Now we are ready to prove Theorem~\ref{t:heart}.

\begin{proof}[Proof of Theorem~\ref{t:heart}]
Fix $\delta\in (0,\eps)$. By Lemmas~\ref{l:dkey} and \ref{l:ub} there exists a random $N\in \N^+$ such that, almost surely,
for all $k>N$ we have
\begin{enumerate}[(i)]
\item \label{i1} $\max_{t\in [0,1]} |B(t)|\leq N$;
\item \label{i2} $\Pi(\iU_k,2\eps)$ holds;
\item \label{i3} $\diam B(I_{k,p})\leq k2^{-\alpha k}$ for all $0\leq p<2^{k}$;
\item \label{i4} $2k+1\leq 2^{\eps k}$.
\end{enumerate}
Fix a sample path $B$ and $N\in \N^+$ for which the above properties hold. Let us fix an arbitrary $k>N$,
it is enough to prove that $\Gamma(\iU_k,3\eps)$ holds. Let $q\in \Z$ and $n\in \N^+$ be given, we need to show that
\begin{equation} \label{eq:Bnm} G_{n,q}(\iU_k)\leq 2^{3\eps n}.
\end{equation}
If $n<k$ then $\iU_k\subset \iI^k$ implies that $G_{n,q}(\iU_k)=0$, and we are done.
Now assume that $n\geq k$. Property \eqref{i1} yields that if $q'\in \Z$ and $|q'|>n2^{\alpha n}$ then $P_{n,q'}(\iU_k)=0$.
Therefore \eqref{i2} implies that for all $q'\in \Z$ we have
\begin{equation} \label{eq:Snm} P_{n,q'}(\iU_k)\leq 2^{2 \eps n}.
\end{equation}
Let $I_{n,p}$ be a time interval of order $n$ such that $B(I_{n,p})\cap J_{n,q}\neq \emptyset$.
By \eqref{i3} we have
\begin{equation} \label{eq:BInp} B(I_{n,p})\subset \bigcup_{q': |q'-q|< n} J_{n,q'}.
\end{equation}
Finally, \eqref{eq:BInp}, \eqref{eq:Snm} and \eqref{i4} imply that
\begin{equation*}
G_{n,q}(\iU_k)\leq \sum_{q': |q'-q|\leq n} P_{n,q'}(\iU_k) \leq (2n+1) 2^{2\eps n} \leq 2^{3\eps n}.
\end{equation*}
Hence \eqref{eq:Bnm} holds, and the proof is complete.
\end{proof}

\section{The main theorems} \label{s:main}

The goal of this section is to prove Theorems~\ref{t:main}, \ref{t:p}, and \ref{t:mixed}.

\begin{proof}[Proof of Theorem~\ref{t:main}] As $B$ is $\gamma$-H\"older continuous for all
$\gamma<\alpha$, Fact~\ref{f:Holder} yields that, almost surely, for all $A\subset [0,1]$ we have
\[\dim_H B(A)\leq (1/\alpha)\dim_H A \quad \textrm{and} \quad \dim_P B(A)\leq (1/\alpha)\dim_P A.\]
Therefore it is enough to show the opposite inequalities. Fix an arbitrary $\eps\in (0,1)$,
it is enough to prove that, almost surely, for all $E\subset \R$ we have
\begin{enumerate}
\item \label{i:H} $\dim_H (B^{-1}(E)\cap D)\leq \alpha \dim_H E+4\eps$;
\item \label{i:P}  $\dim_P (B^{-1}(E)\cap D)\leq \alpha \dim_P E+4\eps$.
\end{enumerate}
As $\dim_{MA} D\leq \alpha$, we have $\dim^{\eps}_{MA} D\leq \alpha$. Therefore
$D=\bigcup_{i=1}^{\infty} D_i$, where $\dim^{\eps}_{A} D_i<\alpha+\eps$ for all $i\in \N^+$.
Thus, by the countable stability of Hausdorff and packing dimensions, we may assume that $\dim^{\eps}_{A} D<\alpha+\eps$.
For all $n\in \N^+$ let
\[\iU_n=\{U\in \iI_n: U\cap D\neq \emptyset\}.\]
Since $\dim^{\eps}_{A} D<\alpha+\eps$, the set $\iU_n\subset \iI^{n}$ is $(\alpha+\eps,\eps)$-balanced for all
$n$ large enough. Therefore Theorem~\ref{t:heart} yields that, almost surely, $\Gamma(\iU_n,3\eps)$ holds for all large enough $n$.
Fix a sample path $B$ and $N\in \N^+$ such that $\Gamma(\iU_n,3\eps)$ holds for all $n\geq N$. Fix an arbitrary $E\subset \R$.

First we prove \eqref{i:H}. Let $\delta>0$ be arbitrary.
Let $\iJ=\bigcup_{n=N}^{\infty} \iJ_n$ and let $s=\dim_H E$,
then there is a cover $E\subset \bigcup_{k=1}^{\infty} J_k$ such that $J_k\in \iJ$ for all $k$ and
$\sum_{k=1}^{\infty} (\diam J_k)^{s+\eps}<\delta$. For all $n\geq N$ let $M_n$ be the number of indices $k$ for which
$J_k\in \iJ_n$, which implies that
\begin{equation} \label{eq:Mn} \sum_{n=N}^{\infty} M_n 2^{-\alpha (s+\eps)n}<\delta.
\end{equation}
The definition of $\Gamma(\iU_n,3\eps)$ yields that for all $n\geq N$ and $J\in \iJ_n$ the set $B^{-1}(J)\cap D$ can be covered by $2^{3\eps n}$
time intervals of length $2^{-n}$. Therefore there is a covering of $B^{-1}(E)\cap D$ containing for each $n\geq N$ at most
$M_n 2^{3\eps n}$ intervals of $\iI_n$. Thus \eqref{eq:Mn} yields that
\[\iH^{\alpha s+4\eps}_{\infty}(B^{-1}(E)\cap D)\leq \sum_{n=N}^{\infty} M_n 2^{3\eps n} 2^{-(\alpha s+4\eps)n}\leq \sum_{n=N}^{\infty} M_n 2^{-\alpha (s+\eps)n}<\delta.\]
As $\delta>0$ was arbitrary, we obtain that $\iH^{\alpha s+4\eps}_{\infty}(B^{-1}(E)\cap D)=0$. Therefore
$\dim_H (B^{-1}(E)\cap D)\leq \alpha s+4\eps$, and \eqref{i:H} follows.

Now we prove \eqref{i:P}. Assume that $\overline{\dim}_M E=t$, first we show that
\begin{equation} \label{eq:uM} \overline{\dim}_M (B^{-1}(E)\cap D)\leq \alpha \, \overline{\dim}_M E+4\eps.
\end{equation}
Fix $n\geq N$, by increasing $N$ if necessary we may assume that $E$ can be covered by $2^{\alpha (t+\eps)n}$ intervals of $\iJ_n$.
Since $\Gamma(\iU_n,3\eps)$ holds, for all $J\in \iJ_n$ the set $B^{-1}(J)\cap D$ can be covered by
$2^{3\eps n}$ intervals of $\iI_n$. Therefore $B^{-1}(E)\cap D$ can be covered by $2^{3\eps n} 2^{\alpha (t+\eps)n}$ intervals of
$\iI_n$ having length $2^{-n}$. Thus $\overline{\dim}_M (B^{-1}(E)\cap D)\leq \alpha t+4\eps$, so \eqref{eq:uM} holds.
Applying this for $E_i$ in place of $E$ we obtain that
\begin{align*} \dim_P (B^{-1}(E)\cap D)&\leq \inf\left\{\sup_{i} \overline{\dim}_M (B^{-1}(E_i)\cap D): E=\bigcup_{i=1}^{\infty} E_i\right\} \\
&\leq \inf\left\{\sup_{i} \alpha \, \overline{\dim}_M E_i+4\eps: E= \bigcup_{i=1}^{\infty} E_i\right\}\\
&=\alpha  \dim_P E+4 \eps.
\end{align*}
Hence \eqref{i:P} holds, and the proof is complete.
\end{proof}

\begin{definition} Assume that $D\subset [0,1]$, $E\subset \R$, and $I\in \iI$. For all $n\in \N^+$ define
\begin{align*}\iU_n(D,I)&=\{U\in \iI_n: U\subset I \textrm{ and } U\cap D\neq \emptyset\}, \\
\iV_n(D,I)&=\{U\in \iU_n(D,I):  (\inter U)\cap D\neq \emptyset\}. \\
\end{align*}
\end{definition}

\begin{proof}[Proof of Theorem~\ref{t:p}]

Implication $\eqref{equi2} \Rightarrow \eqref{equi1}$ follows from Theorem~\ref{t:main}.

Now we prove $\eqref{equi1} \Rightarrow \eqref{equi2}$. Assume to the contrary that \eqref{equi1} holds
and $\dim_{MA} D>\alpha$. By Theorem~\ref{t:reg} we may assume that $D$ is compact. By the definition of the
modified Assouad dimension there exists an $\eps\in (0,1)$ such that $\dim^{\eps}_{MA} D>\alpha+\eps$.
By Lemma~\ref{l:homog}~\eqref{eq:homog2} we may assume that $\dim^{\eps}_{A} (D\cap U)>\alpha+\eps$ for every
open set $U$ which intersects $D$. Therefore $D$ is perfect.
By Lemma~\ref{l:ub} there is random $M\in \N^+$ such that, almost surely,
for all $n\geq M$ and $p\in \{0,\dots,2^n-1\}$ we have
\begin{equation} \label{eq:diam} \diam B(I_{n,p})< n2^{-\alpha n}.
\end{equation}
%%%
Fix a sample path $B$ and $M\in \N^+$ with property \eqref{eq:diam}.
In order to obtain a contradiction it is enough to construct
a compact set $C\subset D$ such that $\dim_P C\geq \eps^2$ and $\dim_P B(C)\leq \eps^2/(\alpha+\eps^2)$.

First we construct $C$. Define $i_0=0$ and $I_0=[0,1]$.
Let $\ell \in \N$ and assume that positive integers
$m_{i_0\dots i_{\ell-1}},n_{i_0\dots i_{\ell-1}},N_{i_0\dots i_{\ell-1}}\in \N^+$ and intervals $I_{i_0\dots i_\ell}$ with
$D\cap \inter I_{i_0 \dots i_\ell}\neq \emptyset$ are already defined
for every $(i_1,\dots,i_\ell)\in \prod_{k=0}^{\ell-1} \{1,\dots,N_{i_0\dots i_{k}}\}$. Note that for $\ell=0$ the only
assumption is $D\cap \inter I_0\neq \emptyset$ which clearly holds.
Since $\dim^{\eps}_{MA} (D\cap \inter I_{i_0\dots i_\ell})>\alpha+\eps$,
there exist positive integers $m=m_{i_0\dots i_\ell}$ and $n=n_{i_0 \dots i_\ell}$ such that
there is an $I_m\in \iI_m$ with
\begin{equation} \label{eq:Q} |\iU_n(D,I_m)|\geq 2^{(\alpha+\eps)(n-m)},
\end{equation}
and we have
\begin{equation} \label{eq:mndef} m\leq (1-\eps)n.
\end{equation}
We define the lexicographical order $\prec$ on
$\Sigma=\bigcup_{n=1}^{\infty} \N^n$ as follows. Let $\prec_n$ be the lexicographical order on $\N^n$, and
for $\sigma\in \Sigma$ let $|\sigma|$ denote the length of $\sigma$.
For $\sigma\in \Sigma$ and $n\leq |\sigma|$ let $\sigma(n)\in \N^n$ denote
the restriction of $\sigma$ to its first $n$ coordinates. Let $\sigma, \theta\in \Sigma$ such that
$\min\{|\sigma|,|\theta|\}=n$. We write $\sigma\prec \theta$ iff either $\sigma(n)\prec_n \theta(n)$ or
$\sigma(n)=\theta(n)$ and $|\sigma|<|\theta|$. By proceeding according to $\prec$
we may assume that if $(j_0,\dots,j_{q})\prec (i_0,\dots,i_{\ell})$ then
\begin{equation} \label{eq:j1} m=m_{i_0 \dots i_{\ell}}>2^{n_{j_0\dots j_{q}}}.
\end{equation}
For every $E\subset \R$ let
\[\iW_{n,m}(E)=\{U\in \iV_n(D,I_m): B(U)\cap E\neq \emptyset\}.\]
Now we define $N_{i_0\dots i_\ell}\in \N^+$ and intervals $I_{i_0 \dots i_{\ell+1}}$ for all $1\leq i_{\ell+1}\leq N_{i_0\dots i_\ell}$.
By \eqref{eq:diam} the diameter of $B(I_m)$ is at most $m2^{-\alpha m}$, so
it can be covered by $m2^{\alpha(n-m)}+2\leq n2^{\alpha (n-m)}$ intervals of $\iJ_n$.
Since $D\cap \inter I_{i_1\dots i_\ell}$ is perfect, there do not exist three
consecutive intervals in $\iU_n(D,I_m)$ such that none of their interior intersects $D$.
Therefore \eqref{eq:Q} and \eqref{eq:mndef} imply that there is an interval
$J\in \iJ_n$ such that
\begin{equation*} |\iW_{n,m}(J) |\geq 2^{\eps(n-m)}/(3n)\geq  2^{(\eps^2+o(1))n}. \end{equation*}
Define $N_{i_0\dots i_\ell}$ and intervals $I_{i_0\dots i_\ell i_{\ell+1}}$ such that
\begin{equation*}
N_{i_0\dots i_\ell}=2^{(\eps^2+o(1))n} \quad \textrm{ and }
\quad \{I_{i_0\dots i_\ell i_{\ell+1}}\}_{1\leq i\leq N_{i_0\dots i_\ell}}\subset \iW_{n,m}(J).
\end{equation*}
Define $\Sigma_0\subset \Sigma$ as
\[\Sigma_0=\bigcup_{\ell=0}^{\infty}\left\{(i_0,\dots,i_{\ell}): 1\leq i_k\leq N_{i_0\dots i_{k-1}} \textrm{ for all } 1\leq k\leq \ell\right\}.\]
Define the compact set $C$ by
\begin{equation*}C=\bigcap_{\ell=1}^{\infty} \left(\bigcup_{i_1=1}^{N_{i_0}} \cdots  \bigcup_{i_\ell=1}^{N_{i_0 \dots i_{\ell-1}}} I_{i_0\dots i_\ell}\right).
\end{equation*}

Now we prove that $\dim_P C\geq \eps^2$. By Lemma~\ref{l:pack}~\eqref{eq:pack1} it is enough to prove that for each open set $U$ intersecting $C$ we have $\overline{\dim}_M (C\cap U)\geq \eps^2$. Fix such an open set $U$, then for every large enough $\ell\in \N^+$ there is an interval $I_{i_0\dots i_\ell}\subset U$ for some $(i_1,\dots,i_\ell) \in \prod_{k=0}^{\ell-1} \{1,\dots,N_{i_0\dots i_{k}}\}$.
The definition of $m=m_{i_0\dots i_\ell}$, $n=n_{i_0\dots i_\ell}$, and $N_{i_0\dots i_\ell}$
yield that
\[N_{2^{-n}}(C\cap U)\geq N_{i_0\dots i_\ell}=2^{(\eps^2+o(1))n}.\]
Clearly $n=n_{i_0\dots i_\ell}\to \infty$ as $\ell \to \infty$, which implies that $\overline{\dim}_M (C\cap U)\geq \eps^2$.

Finally, we prove that
\[\dim_P B(C)\leq \overline{\dim}_M B(C)\leq  \eps^2/(\alpha+\eps^2).\]
For $k\in \N^+$ and $E\subset \R$ let $M_k(E)$ denote the number of intervals of $\iJ_k$ that are needed to cover $E$.
We need to prove that
\begin{equation} \label{eq:mkb} M_k(B(C))\leq 2^{(\alpha \eps^2/(\alpha+\eps^2)+o(1))k}.
\end{equation}
Suppose that $\sigma_0\prec \sigma_1\prec \sigma_2$ are consecutive elements of $\Sigma_0$ and let
$m_i=m_{\sigma_i}$ and $n_i=n_{\sigma_i}$ for $i\in \{0,1,2\}$. We may assume that $m_{1}<k\leq m_{2}$.
By construction, $C\setminus I_{\sigma_1}$ can be covered by at most $|\iI_{n_0}|=2^{n_0}$ intervals of $\iI_{m_2}$.
By \eqref{eq:j1} we have $2^{n_0}<m_1$ and $k\leq m_{2}$, so $C\setminus I_{\sigma_1}$ can be covered by at most $m_1$ intervals of
$\iI_{k}$. Thus \eqref{eq:diam} implies that
\begin{equation} \label{eq:mkb1} M_k(B(C\setminus I_{\sigma_1}))\leq m_1(k+1)<2k^2.
\end{equation}
Now we prove an upper bound for $M_k(B(C\cap I_{\sigma_1}))$. First assume that we have $m_1<k\leq n_1(\alpha+\eps^2)/\alpha$. Then
by the construction there is a $J\in \iJ_{n_1}$ such that for all $I\in \iI_{n_1}$
which intersects $C\cap \inter I_{\sigma_1}$ we have $B(I)\cap J\neq \emptyset$.
Thus \eqref{eq:diam} yields that $B(C\cap \inter I_{\sigma_1})$ can be covered by $3n_1$ consecutive intervals
of $\iJ_{n_1}$. As the contribution of the endpoints of $I_{\sigma_1}$ to $M_k(B(C\cap I_{\sigma_1}))$ is not more than $2$, we have
\begin{equation} \label{eq:mkb3} M_k(B(C\cap I_{\sigma_1}))\leq 4+3n_1 2^{\alpha(k-n_1)} \leq 2^{(\alpha\eps^2/(\alpha+\eps^2)+o(1))k}.
\end{equation}
Finally, assume that $n_1(\alpha+\eps^2)/\alpha<k\leq m_{2}$. Then $I_{\sigma_1}$ contains $2^{(\eps^2+o(1))n_1}$ intervals of
$\iI_k$ which intersects $C$, so \eqref{eq:diam} yields
\begin{equation} \label{eq:mkb4} M_k(B(C\cap I_{\sigma_1})) \leq  (k+1)2^{(\eps^2+o(1))n_1}\leq 2^{(\alpha\eps^2/(\alpha+\eps^2)+o(1))k}.
\end{equation}
Inequalities~\eqref{eq:mkb1}, \eqref{eq:mkb3}, and \eqref{eq:mkb4} imply \eqref{eq:mkb}. The proof is complete.
\end{proof}

Now we prove Theorem~\ref{t:mixed}. First we need some preparation.

\begin{definition} Let $\iU\subset \iI$ be a set of dyadic time intervals. For $\beta>0$ define
\[H^{\beta}(\iU)=\sum_{U\in \iU} (\diam U)^{\beta}.\]
\end{definition}

\begin{lemma} \label{l:balance} Let $\iU\subset \iI$ be a set of dyadic time intervals and let $\beta>0$. Then there is a set $\iV\subset \iI$ such that
\begin{enumerate}
\item \label{i:a} $\iV$ is $\beta$-balanced,
\item \label{i:b} $\bigcup \iU \subset \bigcup \iV$,
\item \label{i:c} $H^{\beta}(\iV)\leq H^{\beta}(\iU)$.
\end{enumerate}
\end{lemma}

\begin{proof}
Let $\iU_0=\iU$, for all $k\in \N$ we inductively define 
\[\iU_{k+1}=\{I\in \iI_m: m\in \N \textrm{ and } N_n(\iU_k,\iI)\geq 2^{\beta (n-m)} \textrm{ for some } n\geq m \}.\]
Taking $n=m$ above shows that $\iU_k\subset \iU_{k+1}$ for all $k\in \N$. Define 
\[ \iU_{\infty}=\bigcup_{k=1}^{\infty} \iU_k  \]
and let $\iV$ be the set of maximal elements of $\iU_{\infty}$ with respect to inclusion, that is, 
\[\iV=\{V\in \iU_{\infty}: \textrm{ there is no } W\in \iU_{\infty} \setminus\{V\} \textrm{ with } V\subset W\}. \]  
First assume to the contrary that $\iV$ is not $\beta$-balanced. Let us choose $m<n$ and $I\in \iI_m$ such that $N_n(\iV,I)>2^{\beta(n-m)}$. Then there exists a $k\in \N^+$ such that $\{V\in \iV\cap \iI_n: V\subset I\}\subset  \iU_{k}$. Therefore $N_n(\iU_{k},I)\geq N_n(\iV,I)>2^{\beta(n-m)}$, so $I\in \iU_{k+1}$ by definition. As $I\in \iU_{\infty}$, the elements of $\{V\in \iV\cap \iI_n: V\subset I\}$ are not maximal in $\iU_{\infty}$ with respect to inclusion, so they cannot be in $\iV$. This is a contradiction, which proves \eqref{i:a}. Clearly $\bigcup \iU\subset \bigcup \iU_{\infty}=\bigcup \iV$, so \eqref{i:b} holds. Finally, we will prove \eqref{i:c}. Let us fix $I\in \iV$. As the intervals in $\iV$ are non-overlapping, it is enough to show that 
\begin{equation} \label{eq:bet} (\diam I)^{\beta} \leq \sum_{U\in \iU,\, U\subset I} (\diam U)^{\beta}.
\end{equation}
Let $k\in \N$ be the minimal number such that $I\in \iU_k$, we prove the claim by induction on $k$. If $k=0$ then \eqref{eq:bet} is straightforward. Assume by induction that \eqref{eq:bet} holds for $k=\ell$, it is enough to prove it for $k=\ell+1$. We can choose $m<n$ such that $I\in \iI_m$ and $N_n(\iU_\ell,I)\geq 2^{\beta(n-m)}$. 
Thus 
\[(\diam I)^{\beta}=2^{-\beta m}\leq N_n(\iU_\ell,I) 2^{-\beta n}=\sum_{V\in \iU_\ell\cap \iI_n,\, V\subset I} (\diam V)^{\beta}.\]
Therefore, using also the induction hypothesis and that the intervals of $\iI_n$ are non-overlapping, we obtain that 
\begin{align*} 
(\diam I)^{\beta}&\leq \sum_{V\in \iU_\ell\cap \iI_n,\, V\subset I} (\diam V)^{\beta} \\
&\leq \sum_{V\in \iU_\ell\cap \iI_n,\, V\subset I} \left(\sum_{U\in \iU,\, U\subset V} (\diam U)^{\beta}\right) \\
&\leq \sum_{U\in \iU,\, \iU\subset I} (\diam U)^{\beta}.
\end{align*}
Thus \eqref{eq:bet} holds, and the proof is complete.
\end{proof}

\begin{proof}[Proof of Theorem~\ref{t:mixed}]
Implication $\eqref{mix1} \Rightarrow \eqref{mix2}$ is straightforward, if \eqref{mix1} holds then
$(1/\alpha)\dim_H D \leq \dim_P(B(D))\leq 1$, which implies $\dim_H D\leq \alpha$.

Now we prove $\eqref{mix2} \Rightarrow \eqref{mix1}$.
Fix $0<\eps<1-\alpha$ arbitrarily, it is enough to show that, almost surely,
for all $E\subset \R$ we have
\begin{equation} \label{eq:eps} \dim_H (B^{-1}(E)\cap D)\leq \alpha \dim_P E+5 \eps.
\end{equation}
%%%
Since $\dim_H D\leq \alpha$, there are covers $\iU_k \subset \iI$ of $D$ such that $H^{\alpha+\eps}(\iU_k)<2^{-k}$ for every $k\in \N^+$.
By Lemma~\ref{l:balance} for each $k$ there is a cover $\iV_k\subset \iI$ of $D$ which is $(\alpha+\eps)$-balanced and
$H^{\alpha+\eps}(\iV_k)<2^{-k}$. As $\alpha+\eps<1$, the inequality $H^{\alpha+\eps}(\iV_k)<2^{-k}$ implies that $\iV_k\subset \iI^k$.
By Theorem~\ref{t:heart}, almost surely, $\Gamma(\iV_k,3\eps)$ holds for all large enough $k$.
Fix a sample path $B$ and $N\in \N^+$ such that $\Gamma(\iV_k,3\eps)$ holds for all $k\geq N$.
Let $E\subset \R$ be arbitrarily fixed, first we prove that
\begin{equation} \label{eq:HP} \dim_H (B^{-1}(E)\cap D)\leq \alpha \, \overline{\dim}_M E+5\eps.
\end{equation}
Assume that $\overline{\dim}_{M} E=t$, by increasing $N$ if necessary we may assume that for all $k\geq N$ the set
$E$ can be covered by $2^{\alpha(t+\eps)k}$ intervals of $\iJ_k$.
Fix $k\geq N$ and for all $n\geq k$ define
\[\iW_n=\{I\in \iV_k\cap \iI_n: B(I)\cap E\neq \emptyset\}.\]
Fix $n\geq k$. Since $\Gamma(\iV_k,3\eps)$ holds, for every $J\in \iJ_n$ we have
\[|\{I\in \iW_n: B(I)\cap J\neq \emptyset\}|\leq 2^{3\eps n}.\]
As $E$ can be covered by $2^{\alpha  (t+\eps)n}$ intervals of $\iJ_n$, the above inequality yields
\[|\iW_n|\leq 2^{3\eps n} 2^{\alpha (t+\eps)n}.\]
Therefore
\begin{equation} \label{eq:4eps}
H^{\alpha t+5\eps}(\iW_n)\leq 2^{3\eps n} 2^{\alpha (t+\eps)n} 2^{-(\alpha t+5\eps)n}\leq 2^{-\eps n}.
\end{equation}
Since $\bigcup_{n=k}^{\infty} \iW_n$ is a covering of $B^{-1}(E)\cap D$, inequality \eqref{eq:4eps} implies that
\begin{equation*} \iH^{\alpha t+5\eps}_{\infty}(B^{-1}(E)\cap D)\leq \sum_{n=k}^{\infty} H^{\alpha t+5\eps}(\iW_n) \leq  \sum_{n=k}^{\infty}  2^{-\eps n}=c_{\eps} 2^{-\eps k},
\end{equation*}
where $c_{\eps}=1/(1-2^{-\eps})$. This is true for all $k$, thus $\iH^{\alpha t+5\eps}_{\infty}(B^{-1}(E)\cap D)=0$. Therefore
$\dim_H (B^{-1}(E)\cap D)\leq \alpha t+5\eps$, so \eqref{eq:HP} holds.

Finally, applying the countable stability of Hausdorff dimension and
\eqref{eq:HP} for $E_i$ in place of $E$ we obtain that
\begin{align*} \label{eq:Q}
\dim_H (B^{-1}(E)\cap D)&=\inf\left\{ \sup_i \dim_H (B^{-1}(E_i)\cap D) : E=\bigcup_{i=1}^{\infty} E_i \right\}  \\
&\leq \inf\left\{ \sup_i \alpha \, \overline{\dim}_M E_i+5 \eps : E=\bigcup_{i=1}^{\infty} E_i \right\}\\
&=\alpha \dim_P E+5 \eps.
\end{align*}
Hence \eqref{eq:eps} holds, and the proof is complete.
\end{proof}

\section{Sets defined by digit restrictions and self-similar sets} \label{s:examples}

The goal of this section is to prove Theorems~\ref{t:digit} and \ref{t:ss}. These answer
Problem~\ref{p:H} in case of sufficiently homogeneous sets. Problem~\ref{p:H} remains open in general.

\subsection{Sets defined by digit restrictions} Before proving Theorem~\ref{t:digit} we need some preparation.

\begin{definition}
Let $S\subset \N^+$. For all $m,n\in \N$ with $m<n$ define
\[d_{m,n}(S)=\frac{|S\cap \{m+1,\dots,n\}|}{n-m} \quad \textrm{and} \quad d_n(S)=\frac{|S\cap \{1,\dots,n\}|}{n}.\]
\end{definition}

\begin{fact} \label{f:d} For all $S\subset \N^{+}$ we have
\begin{align*} \dim_H D_{S}&=\underline{\dim}_M D_S=\liminf_{n\to \infty} d_{n}(S), \\
\dim_{P} D_S&=\overline{\dim}_M D_S=\limsup_{n\to \infty} d_{n}(S), \\
\dim_{MA} D_S&=\lim_{\eps\to 0+} \limsup_{n\to \infty} \max_{m\leq (1-\eps)n} d_{m,n}(S),\\
\dim_A D_S&=\limsup_{n\to \infty} \max_{m\leq n} d_{m,n}(S).
\end{align*}
\end{fact}

\begin{proof} The statements for the Hausdorff and Minkowski dimensions are well known,
the proof of the lower bound for the Hausdorff dimension is based on Frostman's lemma,
while the others follow from an easy computation, similarly to the case of Assouad dimension.
Lemma~\ref{l:pack} yields that the packing dimension agrees with the upper Minkowski dimension. Lemma~\ref{l:homog}~\eqref{eq:homog1} and
straightforward calculation show that for all $0<\eps<1$ we have
\[ \dim^{\eps}_{MA} D_S=\dim^{\eps}_{A} D_S=\limsup_{n\to \infty} \max_{m\leq (1-\eps)n} d_{m,n}(S).\]
Then the definition of modified Assouad dimension completes the proof.
\end{proof}

\begin{definition} \label{d:Nn} Let $D\subset [0,1]$. For a dyadic interval $I\in \iI$ and $n\in \N^+$ define
\[N_n(D,I)=|\{U\in \iI_n: U\subset I \textrm{ and } U\cap D\neq \emptyset\}|.\]
\end{definition}

\begin{lemma} \label{l:tech} Let $D\subset [0,1]$ with $\dim_{MA} D>\alpha$.
Then there exist positive integers $m_k,n_k$ and $\eps\in (0,1)$ such that for all $k\in \N^+$
\begin{enumerate}[(i)]
\item \label{0.1} there is an $I_k\in \iI_{m_k}$ such that $N_{n_k}(D,I_k)\geq 2^{(\alpha+\eps)(n_k-m_k)}$,
\item \label{0.2} $n_k/m_{k+1}=o(1)$,
\item \label{0.3} $m_k=(1-\eps+o(1))n_k$.
\end{enumerate}
\end{lemma}

\begin{proof} As $\dim_{MA} D>\alpha$, we have $\dim^{\delta}_{A} D>\alpha+\delta$ for some $\delta\in (0,1)$.
Thus there exist positive integers $m_k,n_k$ such that \eqref{0.1} and \eqref{0.2} hold with $\delta$ in place of $\eps$,
and $m_k\leq (1-\delta)n_k$. Fix $m=m_k$ and $n=n_k$ and assume that $m<i<n$.
Let $I_i\in \iI_i$ such that $N_{n}(D,I_i)=\max\{N_{n}(D,I): I\in \iI_i\}$. Then clearly
\[N_i(D,I_m) N_n(D,I_i)\geq N_{n}(D,I_m),\]
so either $N_i(D,I_m)\geq 2^{(\alpha+\delta)(i-m)}$ or
$N_n(D,I_i)\geq 2^{(\alpha+\delta)(n-i)}$.
Therefore we can divide the interval $[m,n]$ into two intervals of (almost) equal length and keep one
such that \eqref{0.1} still holds with $\delta$. We iterate this process and stop when our new interval $[m,n]$ satisfies
$m/n\geq 1-\delta$. This works for all large enough $m,n$ and we obtain a constant $c(\delta)<1$ such that $m/n\leq c(\delta)$.
By redefining $m_k,n_k$ and choosing a convergent subsequence of $m_k/n_k$ we may assume that $m_k/n_k\to 1-\eps$,
where $0<1-c(\delta)\leq \eps \leq \delta$. Then $m_k,n_k$ and $\eps\in (0,1)$ satisfy the above properties.
\end{proof}

\begin{definition} Let $\Sigma=\{0,1\}^{\N^+}$ and
let $\Sigma_{*}=\bigcup_{n=1}^{\infty}\Sigma(n)$, where $\Sigma(n)=\{0,1\}^n$.
For $\sigma \in \Sigma\cup \Sigma_{*}$ denote by $\sigma_i$ the $i$th coordinate of $\sigma$ and let
$|\sigma|$ be the length of $\sigma$.
Define $F\colon \Sigma \cup \Sigma_{*}\to [0,1]$ as
\[F(\sigma)=\sum_{i=1}^{|\sigma|} \sigma_i2^{-i},\]
and for all $n\in \N^+$ and $\sigma\in \Sigma(n)$ let
\[I(\sigma)=[F(\sigma),F(\sigma)+2^{-n}]\in \iI_n.\]
For each $S\subset \N^+$ and $n\in \N^+$ let
\begin{align*} \Sigma_{S}&=\{\sigma\in \Sigma: \sigma_i\in \{0,1\} \textrm{ if } i\in S \textrm{ and } \sigma_i=0 \textrm{ if } i\notin S\}, \\
S(n)&=S\cap \{1,\dots,n\}.
\end{align*}
For $\sigma \in \Sigma\cup \Sigma_{*}$ and integers $0\leq m<n\leq |\sigma|$ define \[\sigma(m,n)=(\sigma_{m+1},\dots, \sigma_n)\in \{0,1\}^{n-m} \quad \textrm{and} \quad  \sigma(n)=\sigma(0,n).\]
For $\Lambda \subset \Sigma \cup \Sigma_{*}$ let
\[\Lambda(m,n)=\{\sigma(m,n): \sigma\in \Lambda\}\subset \{0,1\}^{n-m} \quad \textrm{and} \quad  \Lambda(n)=\Lambda(0,n).\]
\end{definition}

\begin{proof}[Proof of Theorem~\ref{t:digit}]
Implication $\eqref{digi2} \Rightarrow \eqref{digi1}$ follows from Theorem~\ref{t:main}.

Now we prove $\eqref{digi1} \Rightarrow \eqref{digi2}$. Suppose that \eqref{digi1} holds.
By Lemma~\ref{l:ub} almost surely there is a
random $M\in \N^+$ such that
for all $n\geq M$ and $p\in \{0,\dots,2^n-1\}$ we have
\begin{equation} \label{eq:diam0} \diam B(I_{n,p})\leq n2^{-\alpha n}.
\end{equation}
%%%
Fix a sample path $B$ and $M\in \N^+$ with this property. Assume to the contrary that $\dim_{MA} D_S>\alpha$ and
$\dim_H D_S=\liminf_{n\to \infty} d_{n}(S)=s>0$,
in order to obtain a contradiction it is enough to construct a compact set $C\subset D_S$ such that $\dim_H B(C)<(1/\alpha) \dim_H C$.

First we construct $C$. By Lemma~\ref{l:tech} there exist positive integers $m_k,n_k$ and $\eps \in (0,1)$ such that
\begin{enumerate}
\item \label{1.1} $d_{m_k,n_k}(S)\geq \alpha+\eps$ for all $k\in \N^+$,
\item \label{1.2} $M<m_k<n_k<m_{k+1}$ for all $k\in \N^+$,
\item \label{1.3} $n_k/m_{k+1}=o(1)$,
\item \label{1.4} $m_k=(1-\eps+o(1))n_k$.
\end{enumerate}

As $\liminf_{n\to \infty} d_n(S)=s$, we can define positive integers $\ell_k$ and
\[T=S\setminus \bigcup_{k=1}^{\infty} (\ell_k,m_k]\subset S\]
such that
\begin{enumerate}[(i)]
\item \label{2.1} $\ell_k\leq m_k<n_k<\ell_{k+1}$ for all $k\in \N^+$,
\item \label{2.2} $\lim_{k\to \infty} d_{m_k}(T)=s$.
\end{enumerate}
Let $k\in \N^+$ and $\sigma\in \Sigma_{T}(m_k)$. Inequality \eqref{eq:diam0} and $m_k>M$ implies that \[\diam B(I(\sigma))\leq m_k2^{-\alpha m_k},\]
so $B(I(\sigma))$ can be covered by $m_k2^{\alpha(n_k-m_k)}+1\leq n_k2^{\alpha(n_k-m_k)}$ intervals
of $\iJ_{n_k}$. As $T\cap (m_k,n_k]=S\cap (m_k,n_k]$, property \eqref{1.1} yields that
\[|\Sigma_{T}(m_k,n_k)|\geq 2^{(\alpha+\eps)(n_k-m_k)}.\]
Therefore there exists a $J=J(\sigma)\in \iJ_{n_k}$ such that
\begin{align*} |\{\lambda \in \Sigma_{T}(n_k): \lambda(m_k)=\sigma,\ B(F(\lambda))\in J\}|&\geq (1/n_k)2^{\eps(n_k-m_k)} \\
&=2^{(\eps-o(1))(n_k-m_k)} \\
&=2^{(\eps^2+o(1))n_k},
\end{align*}
where we used \eqref{1.4} in the last line.
For all $k\in \N^+$ and $\sigma\in \Sigma_{T}(m_k)$ define $\Lambda_k(\sigma)$ and $p_k\in \N^+$ such that
\begin{align*}
\Lambda_k(\sigma)&\subset \{\lambda \in \Sigma_{T}(n_k): \lambda(m_k)=\sigma,\ B(F(\lambda))\in J(\sigma)\}, \\
|\Lambda_k(\sigma)|&=p_k=2^{(\eps^2+o(1))n_k}.
\end{align*}
Define $\Lambda\subset \Sigma$ such that \[\Lambda=\{\sigma \in \Sigma_T: \sigma(n_k)\in \Lambda_k(\sigma(m_k)) \textrm{ for all } k\in \N^+\}.\]
Define the compact set $C\subset D_T$ as
\[C=\{F(\sigma): \sigma \in \Lambda\}=\bigcap_{n=1}^{\infty} \left( \bigcup_{\sigma \in \Lambda(n)} I(\sigma) \right).\]
Let
\[V=T\setminus \bigcup_{k=1}^{\infty} (m_k,n_k]\subset T,\]
first we prove an upper bound for $\dim_H B(C)$. Let $\sigma \in \Lambda(m_k)$.
The definition of $\Lambda_k(\sigma)$ and \eqref{eq:diam0} imply that
$\bigcup_{\lambda\in \Lambda_k(\sigma)} B(I(\lambda))$ can be covered by $2n_k+1$ intervals
of $\iJ_{n_k}$. By \eqref{1.3} we have
\[|\Lambda(m_k)|=2^{(1+o(1))|V(m_k)|},\]
so for any $n_k\leq n\leq m_{k+1}$ the image $B(C)$ can be covered by
\begin{equation*}  |\Lambda(m_k)|(2n_k+1)2^{|V\cap (n_k,n]|}=2^{(d_{n}(V)+o(1))n}
\end{equation*}
intervals of $\iJ_{n}$ having diameter $2^{-\alpha n}$.
Let $W=\N^+\setminus \bigcup_{k=1}^{\infty} (m_k,n_k)$, clearly
for all $m_k<n<n_k$ we have $d_{n_k}(V) \leq d_{n}(V)$. Therefore, as the lower Minkowski dimension is an upper bound for
the Hausdorff dimension, we obtain that
\begin{align}
\begin{split} \label{eq:1-eps}
\dim_H B(C)&\leq \liminf_{n\in W} \frac{\log 2^{(d_{n}(V)+o(1))n}}{\log 2^{\alpha n}}\\
&=(1/\alpha)\liminf_{n\in W}  d_{n}(V) \\
&=(1/\alpha)\liminf_{n\to \infty}  d_{n}(V).
\end{split}
\end{align}
%%%

By \eqref{eq:1-eps} it is enough to show that
\begin{equation} \label{eq:1+eps} \dim_H C> \liminf_{n\to \infty}  d_{n}(V).
\end{equation}
Define a Borel probability measure $\mu$ as follows. For all $k\in \N^+$ and $\sigma \in \Lambda(n_k)$ let
\[\mu(I(\sigma))=\frac{1}{|\Lambda(n_k)|}=\frac{1}{2^{|V(n_k)|}\prod_{i=1}^{k} p_i}.\]
This uniquely defines a $\mu$ with $\supp(\mu)=C$. By Frostman's Lemma we have
\begin{equation} \label{eq:Fr} \dim_H C\geq \sup\{c: \mu(I(\sigma))\leq 2^{-(c+o(1))n} \textrm{ for all } \sigma\in \Lambda(n)\},
\end{equation}
hence in order to find a lower bound for $\dim_H C$ we will estimate $\mu(I(\sigma))$ from above.
Let $n\in \N^+$ and let $\sigma \in \Lambda(n)$ such that $m_k<n\leq m_{k+1}$.

First assume that $m_k<n\leq m_k(1+\eps)$.
Then \eqref{1.3} and \eqref{2.2} yield that
\begin{align}
\begin{split} \label{eq:mu1}
\mu(I(\sigma))&\leq \mu(I(\sigma(m_k)))\leq 2^{-|V(m_k)|}=2^{-(1+o(1))|T(m_k)|}\\
&=2^{-(s+o(1))m_k}\leq 2^{-(s/(1+\eps)+o(1))n}.
\end{split}
\end{align}

Now suppose that $m_k(1+\eps)<n< n_k$. Clearly we have \[|\{\lambda\in \Lambda(n_k): \lambda(n)=\sigma\}|\leq 2^{n_k-n}.\]
Properties \eqref{1.3} and \eqref{2.2}, the definition of $p_k$, and \eqref{1.4} imply that
\[|\Lambda(n_k)|\geq 2^{|V(m_k)|}p_k=2^{(s+o(1))m_k} \cdot 2^{(\eps^2+o(1))n_k}=2^{(s(1-\eps)+\eps^2+o(1))n_k}.\]
Our assumption and \eqref{1.4} yield that
\[n_k\leq (1/(1-\eps^2)+o(1))n.\]
Thus the above three inequalities imply that
\begin{equation} \label{eq:mu2}
\mu(I(\sigma))\leq \frac{2^{n_k-n}}{|\Lambda(n_k)|}\leq 2^{n_k(1-\eps^2-s(1-\eps)+o(1))-n}\leq 2^{-(s/(1+\eps)+o(1))n}.
\end{equation}

Finally, assume that $n_k\leq n\leq m_{k+1}$, then the definition of $p_k$ implies that
\begin{equation} \label{eq:mu3} \mu(I(\sigma))=\frac{1}{2^{|V(n)|}\prod_{i=1}^{k} p_i}\leq \frac{1}{2^{|V(n)|}p_k}=2^{-(d_{n}(V)n+(\eps^2+o(1))n_k)}.
\end{equation}

Inequalities \eqref{eq:mu1}, \eqref{eq:mu2}, \eqref{eq:mu3}, and Frostman's lemma yield that
\begin{equation} \label{eq:Clb} \dim_H C\geq \min \left\{\frac{s}{1+\eps},~ \liminf_{k\to \infty} \min_{n\in [n_k,m_{k+1}]} \left(d_{n}(V)+\eps^2 \frac{n_k}{n}\right)\right\}.
\end{equation}

Properties \eqref{2.2}, \eqref{1.3}, and \eqref{1.4} imply that
\begin{equation} \label{eq:limdn}
\liminf_{n\to \infty} d_{n}(V)\leq \liminf_{k\to \infty} d_{n_k}(V)=\liminf_{k\to \infty} \frac{(s+o(1))m_k}{n_k}=s(1-\eps).
\end{equation}
Inequalities \eqref{eq:Clb}, \eqref{eq:limdn} and $s/(1+\eps)>s(1-\eps)$ yield that it is enough to prove
for \eqref{eq:1+eps} that
\begin{equation} \label{eq:dnv} \liminf_{k\to \infty}  \min_{n\in [n_k,m_{k+1}]} \left(d_{n}(V)+\eps^2 \frac{n_k}{n}\right)>\liminf_{n\to \infty} d_{n}(V).
\end{equation}
Let $n_k\leq n\leq m_{k+1}$. First assume that $\ell_{k+1}\leq n\leq m_{k+1}$. As $V\cap (\ell_{k+1},m_{k+1}]=\emptyset$,
properties \eqref{1.3} and \eqref{2.2} imply that
\begin{equation} \label{eq:lnm}
d_{n}(V)\geq d_{m_{k+1}}(V)=(1+o(1))d_{m_{k+1}}(T)=s+o(1).
\end{equation}
Now suppose that $2n_k/(\eps s)\leq n\leq \ell_{k+1}$. Then $V\cap (n_k,n]=S\cap (n_k,n]$ and $\liminf_{n\to \infty} d_n(S)=s$ imply that
\begin{equation} \label{eq:nnl}
d_{n}(V)\geq d_n(S)-n_k/n\geq s(1-\eps/2)+o(1).
\end{equation}
Finally, assume that $n_k\leq n\leq 2n_k/(\eps s)$. Then clearly
\begin{equation} \label{eq:nnn}
d_n(V)+\eps^2 \frac{n_k}{n}\geq d_n(V)+\eps^3 s/2.
\end{equation}
Inequalities \eqref{eq:lnm}, \eqref{eq:nnl}, \eqref{eq:nnn}, and \eqref{eq:limdn} imply that
\[\liminf_{k\to \infty}  \min_{n\in [n_k,m_{k+1}]} \left(d_{n}(V)+\eps^2 \frac{n_k}{n}\right)\geq \liminf_{n\to \infty} d_{n}(V)+\eps^3 s/2.\]
Hence \eqref{eq:dnv} holds, and the proof is complete.
\end{proof}

\subsection{Self-similar sets}

The goal of this subsection is to prove Theorem~\ref{t:ss}.

\begin{proof}[Proof of Theorem~\ref{t:ss}]
Implication $\eqref{i:ss2} \Rightarrow \eqref{i:ss1}$ follows from Theorem~\ref{t:main}.

Now we prove $\eqref{i:ss1} \Rightarrow \eqref{i:ss2}$. Assume to the contrary that $\dim_{MA} D>\alpha$ and \eqref{i:ss1} hold.
By Lemma~\ref{l:ub} almost surely there is a
random $M\in \N^+$ such that
for all $n\geq M$ and $p\in \{0,\dots,2^n-1\}$ we have
\begin{equation} \label{eq:dia} \diam B(I_{n,p})\leq 2^{-(\alpha+o(1))n}.
\end{equation}
Fix a sample path $B$ and $M\in \N^+$ with this property. Let us recall Definition~\ref{d:Nn}.
We show that there are positive integers
$m_k,n_k,\ell_k,d_k$ and $\eps\in (0,1)$ such that for all $k\in \N^+$
\begin{enumerate}[(1)]
\item \label{3.1} there is an $I_k\in \iI_{m_k}$ such that $N_{n_k}(D,I_k)\geq 2^{(\alpha+\eps+o(1))(n_k-m_k)}$,
\item \label{3.2} $m_k=(1-\eps+o(1))n_k$,
\item \label{3.3} $\sum_{i=1}^{k-1} n_i=o(n_{k})$ and $m_1>M$,
\item \label{3.4}  $n_{k}=\ell_{k} d_{k}$, where $d_1=n_1$, $d_2=n_2$, and $d_k=n_{k-2}=o(n_{k-1})$ for $k>2$.
\end{enumerate}
Indeed, $\dim_{MA} D>\alpha$ and Lemma~\ref{l:tech} imply \eqref{3.1} and \eqref{3.2}. Property~\eqref{3.3} may be assumed by passing to a subsequence.
Adding at most $n_{k-2}$ to $n_k$ does not change the earlier asymptotes, so we may suppose that $n_k$ is divisible by $n_{k-2}$ for $k>2$, so
$\ell_1=\ell_2=1$ and $\ell_k=n_k/n_{k-2}$ for $k>2$ satisfies \eqref{3.4}.

Properties \eqref{3.1} and \eqref{3.2} imply that $\overline{\dim}_M D\geq (\alpha+\eps)\eps>\eps^2$.
Let $\overline{\dim}_M D=t>0$. It is enough to construct a compact set $C\subset D$
such that $\dim_H B(C)\leq t/(2\alpha)$ and $\dim_H C\geq t/(2-\eps^2)>t/2$.

First we construct $C$. Assume that $D=\bigcup_{i=1}^{k_0} f_i(D)$, where $f_i$ are contractive similarities of $\R$.
Let $r=\min\{\Sim(f_i): i\leq k_0\}$, where $\Sim(f)$ denotes the similarity ratio of $f$.
It is easy to show that for each $x\in D$ and $R\in (0,1)$ there exists a similarity $f\colon D\to B(x,R)\cap D$ such that
$r R\leq \Sim(f)\leq R$.

First we prove that for all $k\in \N^+$ we can define $p_k\in \N^+$ and similarities $\phi^{k}_i\colon D\to D\cap I_k$
for $1\leq i\leq p_k$ such that
\begin{enumerate}[(i)]
\item \label{4.1} $p_k=2^{(\alpha+\eps+o(1))(n_k-m_k)}$,
\item \label{4.2} $r2^{-n_k}\leq \Sim(\phi^{k}_i)\leq 2^{-n_k}$ for all $i$,
\item \label{4.3} $\dist(\phi^k_i(D),\phi^k_j(D))\geq 2^{-n_k}$ for all $i\neq j$.
\end{enumerate}
Indeed, by \eqref{3.1} for each $k$ there exists $p_k$ with \eqref{4.1} such that
there are points $\{x^{k}_i\}_{i=1}^{p_k}$ in $D\cap I_k$ with
$|x^{k}_i-x^{k}_j|\geq 2^{-n_k+2}$ for all $i\neq j$ and $\dist(\{x^{k}_j\},\partial (I_k))\geq 2^{-n_k}$.
Hence for each $k\in \N^+$ and $1\leq i\leq p_k$ there is a similarity
$\phi^k_i\colon D\to D\cap B(x^k_i,2^{-n_k})$ satisfying property \eqref{4.2}. Clearly \eqref{4.3} holds, too.

By \cite[Theorem~4]{F} we have $\underline{\dim}_M D=\overline{\dim}_M D$. Therefore, similarly as above,
we can define positive integers $q_k$ and similarities $\psi^{k}_i\colon D\to D$ for $1\leq i\leq q_k$ such that
\begin{enumerate}[(A)]
\item \label{5.1} $q_k=2^{(t+o(1))d_k}$,
\item \label{5.2} $r2^{-d_k}\leq \Sim(\psi^{k}_i)\leq 2^{-d_k}$ for all $i$,
\item \label{5.3} $\dist(\psi^k_i(D),\psi^k_j(D))\geq 2^{-d_k}$ for all $i\neq j$.
\end{enumerate}
For all $k\in \N^+$ and $\sigma=(j_1,\dots, j_{\ell_k})\subset \{1,\dots,q_k\}^{\ell_k}$ define similarities $\Psi^{k}_{\sigma} \colon D\to D$ as
\[\Psi^{k}_{\sigma}=\psi^k_{j_1} \circ \dots \circ \psi^k_{j_{\ell_k}}.\]

Assume that $k\in \N^{+}$, $\theta=(\sigma_1,i_1,\sigma_2, \dots,i_{k-1},\sigma_k)$, and $i\in \{1,\dots, p_k\}$ are given, where
\[(\sigma_1,\dots,\sigma_{k})\in \prod_{j=1}^{k} \{1,\dots,q_{j}\}^{\ell_j} \quad \textrm{and} \quad
(i_1,\dots,i_{k-1})\in \prod_{j=1}^{k-1} \{1,\dots,p_j\}.\]
Define the similarities $\Phi_{\theta}, \Phi_{\theta i}\colon D\to D$ such that
\begin{equation*} \Phi_{\theta}= \Psi^1_{\sigma_1} \circ \phi^1_{i_1} \circ \dots \circ \phi^{k-1}_{i_{k-1}} \circ \Psi^k_{\sigma_k}
\quad \textrm{and} \quad \Phi_{\theta i}=\Phi_{\theta}\circ \phi^{k}_{i}.
\end{equation*}
\begin{statement} \label{st} For $\theta=(\sigma_1,i_1,\sigma_2, \dots,i_{k-1},\sigma_k)$
we can define $\Sigma_\theta\subset \{1,\dots,p_k\}$ such that
\begin{equation}\label{Sdef} |\Sigma_\theta|=s_{k}=2^{(\eps+o(1))(n_{k}-m_{k})}=2^{(\eps^2+o(1))n_{k}},
\end{equation}
and the similarities $\Phi_{\theta i}$ satisfy
\begin{equation} \label{eq:uni}
\diam \left(\bigcup_{i\in \Sigma_{\theta}} B(\Phi_{\theta i}(D))\right)\leq 2^{-(2\alpha+o(1))n_{k}}.
\end{equation}
\end{statement}
\begin{proof}[Proof of Statement~\ref{st}]
As $\phi^{k}_i(D)\subset I_k$, for all $1\leq i\leq p_k$ we have
\begin{equation} \label{eq:Pik} \Phi_{\theta i} (D)\subset \Phi_{\theta}(I_k).
\end{equation}
Properties \eqref{4.2}, \eqref{5.2}, \eqref{3.4}, and \eqref{3.3} imply that
\begin{equation} \label{eq:Phth} \diam \Phi_{\theta}(I_k)\leq 2^{-2(n_1+\dots+n_{k-1})-n_k-m_k}=
2^{-(1+o(1))(n_k+m_k)}.
\end{equation}
Then \eqref{eq:Pik}, \eqref{eq:Phth}, and \eqref{eq:dia} yield that
\[\diam \left(\bigcup_{i=1}^{p_k} B(\Phi_{\theta i}(D))\right)\leq \diam B(\Phi_{\theta}(I_k))\leq 2^{-(\alpha+o(1))(n_k+m_k)}.\]
Thus $\bigcup_{i=1}^{p_k} B(\Phi_{\theta i}(D))$ can intersect at most $2^{(\alpha+o(1))(n_k-m_k)}$ value intervals of
$J_{2n_k}$. Hence \eqref{4.1} and \eqref{3.2} imply that there is an interval $J\in \iJ_{2n_k}$ such that
\begin{align*} |\{1\leq i \leq p_k: B(\Phi_{\theta i}(D))\cap J\neq \emptyset\}|&\geq p_k/2^{(\alpha+o(1))(n_k-m_k)} \\
&\geq 2^{(\eps+o(1))(n_k-m_k)} \\
&=2^{(\eps^2+o(1))n_{k}}.
\end{align*}
Choose $\Sigma_{\theta}\subset \{1\leq i \leq p_k: B(\Phi_{\theta i}(D))\cap J\neq \emptyset\}$ according to \eqref{Sdef},
we need to prove \eqref{eq:uni}. Similarly as above for all $1\leq i\leq p_k$ we have
\[\diam \Phi_{\theta i}(D)\leq 2^{-2(n_1+\dots +n_k)}=2^{-(2+o(1))n_k}.\]
Therefore by \eqref{eq:dia} for all for all $1\leq i\leq p_k$ we obtain
\begin{equation} \label{eq:BD} \diam B(\Phi_{\theta i}(D))\leq 2^{-(2\alpha+o(1))n_k}.
\end{equation}
As the images $\{B(\Phi_{\theta i}(D))\}_{i\in \Sigma_{\theta}}$ intersect the same $J$ which has length
$2^{-2\alpha n_k}$, inequality \eqref{eq:BD} implies \eqref{eq:uni}. The proof of the statement is complete.
\end{proof}

Now we return to the proof of Theorem~\ref{t:ss}. Define
\[\Theta\subset \prod_{k=1}^{\infty} \left(\{1,\dots,q_k\}^{\ell_k}\times \{1,\dots,p_k\}\right)\quad \textrm{as}\]
\[\Theta=\{(\sigma_1,i_1,\sigma_2,\dots): i_k\in \Sigma_{\sigma_1 i_1 \dots i_{k-1} \sigma_k} \textrm{ for all } k\in \N^+\}.\]
For all $k\in \N^+$ let \[\Theta(k)=\{\theta_k=(\sigma_1,i_1,\dots,\sigma_k,i_k): \exists \theta \in \Theta \textrm{ which extends } \theta_k\}.\]
Define the compact set
\[C=\bigcap_{k=1}^{\infty} \left(\bigcup_{\theta\in \Theta(k)} \Phi_{\theta}(D)\right).\]
Now we prove that $\dim_H B(C)\leq t/(2\alpha)$. By \eqref{4.1} and \eqref{3.3} we have
\begin{equation} \label{eq:sump} \log_2 (p_1 \cdots p_{k-1})\leq (\alpha+\eps+o(1))\sum_{i=1}^{k-1} n_i=o(n_k).
\end{equation}
By \eqref{5.1}, \eqref{3.4}, and \eqref{3.3} we have
\begin{equation} \label{eq:sumq} \log_2 (q_1^{\ell_1} \cdots q_{k}^{\ell_k})=(t+o(1))\sum_{i=1}^{k} n_i=(t+o(1))n_k.
\end{equation}
By \eqref{eq:uni} we obtain that $B(C)$ can be covered by $(\prod_{i=1}^{k-1} q_i^{\ell_i}p_i)q_k^{\ell_k}$ intervals of
length $2^{-(2\alpha+o(1))n_k}$. Since the lower Minkowski dimension is an upper bound for the Hausdorff dimension,
asymptotes \eqref{eq:sump} and \eqref{eq:sumq} imply that 
\[\dim_H B(C)\leq \liminf_{k\to \infty} \frac{\log_2 (\prod_{i=1}^{k-1} p_i \prod_{i=1}^{k} q_i^{\ell_i})}{\log_2 2^{(2\alpha+o(1))n_k}}=t/(2\alpha).\]
Finally, we show that $\dim_H C\geq t/(2-\eps^2)>t/2$. Assume that $k\in \N^+$, $\theta\in \Theta(k)$, $0\leq \ell\leq \ell_{k+1}$ and
$(j_1,\dots, j_\ell)\in \{1,\dots,q_{k+1}\}^{\ell}$. Then define $\Phi_{\theta j_1\dots j_\ell}\colon D \to D$ as
\[\Phi_{\theta j_1\dots j_\ell}=\Phi_{\theta}\circ \psi^{k+1}_{j_1} \circ \dots \circ \psi^{k+1}_{j_\ell}.\]
Each $N\in \N$ can be uniquely written as $N=\left(\sum_{i=1}^k (\ell_i+1)\right)+\ell$, where $k\in \N$ and $0\leq \ell \leq \ell_{k+1}$
depend on $N$. Consider the cover
\[C\subset \bigcup_{\theta\in \Theta(k)} \bigcup_{j_1=1}^{q_{k+1}}\dots \bigcup_{j_\ell=1}^{q_{k+1}} \Phi_{\theta j_1\dots j_\ell}(D),\]
we call $C\cap \Phi_{\theta j_1\dots j_\ell}(D)$ the \emph{elementary pieces of $C$ of level N}.
Every elementary piece of $C$ of level $N-1$ has $c_N$ children, where
\[c_N=
\begin{cases} s_k & \textrm{ if } \ell=0, \\
q_{k+1} & \textrm{ if } \ell>0.
\end{cases}\]
Inequalities \eqref{4.2}, \eqref{4.3}, \eqref{5.2}, \eqref{5.3}, and
\eqref{3.4} imply that the distance between any two elementary pieces of $C$ of level $N$ is at least
\[\eps_N=r^{N} 2^{-2(n_1+\dots +n_k)-\ell d_{k+1}}.\]
We have $\eps_N \searrow 0$ as $N\to \infty$, so \cite[Example~4.6]{F2}  implies that
\begin{equation} \label{eq:lbound} \dim_H C\geq \liminf_{N\to \infty} \frac{\log_2(c_1\cdots c_{N-1})}{-\log_2(c_N \eps_N)}.
\end{equation}
Hence we need to bound the above limes inferior from below.
We use the notation $a_n\sim b_n$ if $a_n/b_n\to 1$ as $n\to \infty$.
By \eqref{Sdef} and \eqref{5.1} we obtain that
\[\log_2 c_N \sim
\begin{cases} \eps^2n_k & \textrm{ if } \ell=0, \\
td_{k+1} & \textrm{ if } \ell>0.
\end{cases}\]
Hence asymptotes \eqref{3.3} and \eqref{3.4} yield that
\[\log_2(c_1 \cdots c_{N-1})\sim
\begin{cases} tn_k & \textrm{ if } \ell=0, \\
(t+\eps^2)n_k+(\ell-1)td_{k+1} & \textrm{ if } \ell>0.
\end{cases}\]
By \eqref{3.3} and \eqref{3.4} we have $\sum_{i=1}^{k} n_i \sim n_k$ and $N=o(n_k+\ell d_{k+1})$, so
\[\log_2 \eps_N=N \log_2 r-2(n_1+\dots +n_k)-\ell d_{k+1} \sim -(2n_k+\ell d_{k+1}).\]
Therefore \eqref{eq:lbound} yields that
\begin{equation} \label{eq:HC}
\dim_H C\geq \liminf_{k\to \infty}\min\left\{\frac{t}{2-\eps^2}, \min_{1\leq \ell \leq \ell_{k+1}} \frac{(t+\eps^2)n_k+(\ell-1)td_{k+1}}{2n_k+\ell d_{k+1}-td_{k+1}}\right\}.
\end{equation}
It is easy to check that for all positive numbers $a,b,c,d$ we have
\begin{equation*} \label{eq:abcd}
\frac{a+b}{c+d}\geq \min\left\{a/c, b/d\right\},
\end{equation*}
and applying this for $a=(t+\eps^2)n_k-td_{k+1}$, $b=\ell t d_{k+1}$, $c=2n_k-td_{k+1}$,
and $d=\ell d_{k+1}$ together with $d_{k+1}=o(n_k)$ implies that for all $\ell>0$ we have
\begin{equation} \label{min}
\frac{(t+\eps^2)n_k+(\ell-1)td_{k+1}}{2n_k+\ell d_{k+1}-td_{k+1}}\geq \min\left\{\frac{(t+\eps^2)n_k-td_{k+1}}{2n_k-td_{k+1}}, t \right\}\sim \frac{t+\eps^2}{2},
\end{equation}
where we used that $t>\eps^2$. Then \eqref{eq:HC} and \eqref{min} yield that
\[\dim_H C\geq \min\left\{\frac{t}{2-\eps^2}, \frac{t+\eps^2}{2} \right\}=\frac{t}{2-\eps^2},\]
and the proof is complete.
\end{proof}

\section{A restriction theorem for fractional Brownian motion} \label{s:restriction}

The main goal of this section is to give a new proof for Theorem~\ref{t:ABMP} based on
Theorem~\ref{t:main}. First we need some preparation.

\begin{definition} Let $\gamma\in (0,1)$, we construct a random compact set $\Gamma(\gamma)\subset [0,1]$ as follows. We keep each interval $I\in \iI_1$ with probability $p=2^{-\gamma}$. Let $\Delta_1\subset \iI_1$ be the the collection of kept intervals and let $S_1=\bigcup \Delta_1$ be their union. If $\Delta_n\subset \iI_n$ and $S_n$ are already defined, then we keep every interval $I\in \iI_{n+1}$ for which $I\subset S_n$ independently with probability $p$. We denote by $\Delta_{n+1}\subset \iI_{n+1}$ the collection of kept intervals and by $S_{n+1}=\bigcup \Delta_{n+1}$ their union. We define a \emph{percolation limit set} as
\[\Gamma(\gamma)=\bigcap_{n=1}^{\infty} S_n.\]
\end{definition}

The following theorem is due to Hawkes \cite[Theorem~6]{H}.

\begin{theorem}[Hawkes] \label{t:Hawkes} Let $\gamma\in (0,1)$ and let $C\subset [0,1]$ be a compact set with $\dim_H C>\gamma$. Then
$\dim_H (C\cap \Gamma(\gamma))>0$ with positive probability.
\end{theorem}

The following theorem follows from a result of Athreya \cite[Theorem~4]{At}.

\begin{theorem}[Athreya] \label{t:A} Let $\{Z_n\}_{n\geq 1}$ be a Galton-Watson branching process such that $\E Z_1=m>1$ and $\E e^{\theta Z_1}<\infty$ for some $\theta>0$. Then there exist $c_1,c_2\in \R^+$ such that for all $n\in \N^+$ and $k>0$ we have
\[\P(Z_n\geq km^n)\leq c_1e^{-c_2k}.\]
\end{theorem}

\begin{remark} Note that the above theorem is proved in \cite{At} under the assumption $\P(Z_1=0)=0$, but
we may assume this by applying the Harris-transformation. For more on this theory see \cite{AN}.
\end{remark}

Fraser, Miao, and Troscheit~\cite[Theorem~5.1]{FMT} proved that $\dim_A \Gamma(\gamma)=1$ almost surely, provided $\Gamma(\gamma)\neq \emptyset$.
The following theorem claims that the modified Assouad dimension behaves differently, we have
$\dim_{MA}\Gamma(\gamma)=\dim_H \Gamma(\gamma)$ almost surely.

\begin{theorem} \label{t:H=MA} Let $\gamma\in (0,1)$. Then
\[\dim_{MA} \Gamma(\gamma)=\dim_H \Gamma(\gamma)=1-\gamma\]
almost surely, provided $\Gamma(\gamma)\neq \emptyset$.
\end{theorem}

\begin{proof}
It is well known that $\dim_H \Gamma(\gamma)=1-\gamma$ almost surely, provided $\Gamma(\gamma)\neq \emptyset$, see e.g.\
\cite[Theorem~2]{H}. By Fact~\ref{f:3d} it is enough to prove that, almost surely, we have $\dim_{MA} \Gamma(\gamma)\leq 1-\gamma$.
Let $0<\eps<1$ be arbitrarily fixed, it is enough to show that $\dim^{\eps}_{A} \Gamma(\gamma)\leq 1-\gamma$ with probability one.
Let $m\in \N^+$ and $I\in \iI_m$. For all $n>m$ let
\[N_n(I)=|\{J\in \iI_n: J\subset I \textrm{ and } J\in \Delta_n\}|\]
and define the event
\[\iA_n=\{N_n(I)\leq n^2 2^{(1-\gamma)(n-m)} \textrm{ for all } m\leq (1-\eps)n \textrm{ and } I\in \iI_m)\}.\]
It is enough to prove that, almost surely, $\iA_n$ holds for all large enough $n$.
Let $Z_n=|\Delta_n|$ for all $n\in \N^+$, then $\{Z_n\}_{n\geq 1}$ is a Galton-Watson branching process with $\E Z_1=2^{1-\gamma}>1$.
Clearly $\E e^{Z_1}<\infty$, so by Theorem~\ref{t:A} there are $c_1,c_2\in \R^+$ such that for all $n\in \N^+$ and $k\in \R^+$ we have
\begin{equation} \label{eq:k} \P(Z_n\geq k 2^{(1-\gamma)n})\leq c_1 e^{-c_2k}.
\end{equation}
For a given $I\in \iI_m$, provided $I\in \Delta_m$, the random variable $N_n(I)$ has the same distribution as $Z_{n-m}$.
Hence \eqref{eq:k} with $k=n^2$ implies that
\[\P(N_n(I)\geq n^2 2^{(1-\gamma)(n-m)})\leq c_1 e^{-c_2n^2}.\]
The number of pairs $(m,I)$ for which $m\leq (1-\eps)n$ and $I\in \iI_m$ is at most $n2^n$, so the probability of
the complement of $\iA_n$ satisfies
\[\P(\iA^c_n)\leq c_1n2^{n}e^{-c_2 n^2}.\]
Therefore $\sum_{n=1}^{\infty} \P(\iA^c_n)<\infty$,
and the Borel-Cantelli lemma yields that $\iA^c_n$ holds only for finitely many $n$. This completes the proof.
\end{proof}

\begin{definition} Let $(\mathcal{K},d_{H})$ be the set of compact subsets of
$[0,1]$ endowed with the \emph{Hausdorff metric}, that is, for each $K_1,K_2\in \mathcal{K}\setminus \{\emptyset\}$ we have
%%%%%%%%%%%%%%%%
\[d_{H}(K_1,K_2)=\min \left\{r: K_1\subset B(K_2,r) \textrm{ and } K_2\subset B(K_1,r)\right\},\]
where $B(A,r)=\{x\in \R: \exists y\in A \textrm{ such that } |x-y|\leq r\}$.
Let $d_H(\emptyset, \emptyset)=0$ and $d_H(K,\emptyset)=1$ for all $K\in \iK\setminus \{\emptyset\}$.
Then $(\mathcal{K},d_{H})$ is a compact metric space, see \cite[Theorem~4.26]{Ke}. Let $C[0,1]$ denote the set of continuous functions
$f\colon [0,1]\to \R$ endowed with the maximum norm. For $\gamma\in (0,1)$ and $n\in \N^+$
let $C^{\gamma}[0,1]$ and $C^{\gamma}_{n}[0,1]$ be the set of functions $f\in C[0,1]$ such that $f$ is
$\gamma$-H\"older continuous and $\gamma$-H\"older continuous with H\"older constant at most $n$, respectively.
For $E\subset X\times Y$ and $x\in X$ let $E_x=\{y\in Y: (x,y)\in E\}$, and for $y\in Y$ let $E^{y}=\{x\in X: (x,y)\in E\}$.
\end{definition}

\begin{lemma} \label{l:Borel} Assume that $0<\gamma<\alpha<1$ and let
\begin{align*} \Delta=\{&(f,K)\in C^{\gamma}[0,1]\times \iK: \textrm{there exist } C\in \iK \textrm{ and } \beta>\alpha \textrm{ such that} \\
&\dim_H C>1-\alpha \textrm{ and } \dim_H (K\cap C)>0, \textrm{ and } f|_{C} \textrm{ is $\beta$-H\"older}\}.
\end{align*}
Then $\Delta$ is a Borel set in $C[0,1]\times \iK$.
\end{lemma}

\begin{proof}
For all $n\in \N^+$ define
\begin{align*}
\Delta_n=\{&(f,K)\in C^{\gamma}_n[0,1]\times \iK: \textrm{ there is a } C\in \iK \textrm{ such that}\\
&\iH^{1-\alpha+1/n}_{\infty}(C)\geq 1/n \textrm{ and } \iH^{1/n}_{\infty}(K\cap C)\geq 1/n, \textrm{ and} \\
&f|_{C} \textrm{ is $(\alpha+1/n)$-H\"older with H\"older constant at most $n$}\}.
\end{align*}
Since $C_n^{\gamma}[0,1]$ and $\iK$ are compact, it is easy to verify that $\Delta_n$ is compact for each $n\in \N^+$.
Clearly $\Delta=\bigcup_{n=1}^{\infty} \Delta_n$, thus $\Delta$ is $\sigma$-compact, so it is a Borel set.
\end{proof}

Now we are ready to give a new proof for Theorem~\ref{t:ABMP}.

\begin{proof}[Proof of Theorem~\ref{t:ABMP}] Assume that $0<\gamma<\alpha<1$ are fixed and
let $\Delta=\Delta(\gamma,\alpha)$ be the Borel set of Lemma~\ref{l:Borel}. First assume that
$(f,K)\in \Delta$. Then there exist $C\in \iK$ and $\beta>\alpha$ such that
$f|_{C}$ is $\beta$-H\"older continuous and $\dim_H(K\cap C)>0$.
By Fact~\ref{f:Holder} the set $E=K\cap C\subset K$ satisfies
\begin{equation} \label{eq:hfe1} \dim_H f(E)\leq (1/\beta) \dim_H E<(1/\alpha) \dim_H E.
\end{equation}
Let $\mu$ and $\nu$ be the distributions of $B$ on $C[0,1]$ and
of $\Gamma(1-\alpha)$ on $\iK$, respectively.
By Theorem~\ref{t:H=MA} we have $\dim_{MA} K\leq \alpha$ for $\nu$ almost every $K$.
Fix such a $K$, then Theorem~\ref{t:main} implies that for $\mu$ almost every $f$ for
all $E\subset K$ we have
\begin{equation} \label{eq:hfe2}
\dim_H f(E)=(1/\alpha) \dim_H E.\end{equation}
Thus \eqref{eq:hfe1} and \eqref{eq:hfe2} imply that $\mu(\Delta^{K})=0$.
As $\Delta$ is Borel, Fubini's theorem yields that $(\mu\times \nu)(\Delta)=0$.

As $(\mu\times \nu)(\Delta)=0$, Fubini's theorem and the fact that $B$ is $\gamma$-H\"older continuous almost surely imply
that $\nu(\Delta_f)=0$ for $\mu$ almost every $f$. Fix such an
$f$ and assume to the contrary that
there is a set $C\subset [0,1]$ such that $f|_{C}$ is $\beta$-H\"older continuous for some $\beta>\alpha$ and
$\dim_H C>1-\alpha$. As $f$ is still $\beta$-H\"older continuous on the closure of $C$, we may assume that $C$ is compact.
Then clearly $\{K\in \iK: \dim_H(K\cap C)>0\}\subset \Delta_f$, thus Theorem~\ref{t:Hawkes} yields that
\[\nu(\Delta_f)\geq \nu(\{K\in \iK: \dim_H(K\cap C)>0\})>0.\]
This is a contradiction, and the proof is complete.
\end{proof}

\subsection*{Acknowledgments}
We are indebted to Jinjun Li for pointing out a mistake in an earlier version of the paper.

\end{document}